\documentclass[12pt,a4paper]{article}
\usepackage{latexsym,amssymb,amsfonts,amsmath,amsthm,nccmath,float,enumitem,setspace,bm,authblk}
\usepackage[usenames,dvipsnames]{xcolor}
\usepackage[hidelinks]{hyperref}
\usepackage[margin=1.75cm]{geometry}
\usepackage{bm}
\usepackage{booktabs}
\usepackage{verbatim}
\usepackage{diagbox}
\usepackage{makecell}
\usepackage{float}
\allowdisplaybreaks[1]

\hypersetup{
	colorlinks,
	linkcolor={red!80!black},
	citecolor={blue!80!black},
	urlcolor={blue!80!black}
}

\newtheorem{thm}{Theorem}[section]

\newtheorem{pro}[thm]{Proposition}

\newtheorem{defi}[thm]{Definition}

\newtheorem{rem}[thm]{Remark}
\newtheorem{lem}[thm]{Lemma}
\newtheorem{core}[thm]{Corollary}

\newtheorem*{statement}{Statement}

\newcommand{\F}{\mathbb{F}}

\def \leq {\leqslant}
\def \geq {\geqslant}

\def\diag{{\rm diag}}

 \def\GL{{\rm GL}}

\def\PSL{{\rm PSL}}
\def\GL{{\rm GL}} \def\SL{{\rm SL}}

 \def\F{{\mathbb F}}

\let\oldproofname=\proofname
\renewcommand{\proofname}{\rm\bf{\oldproofname}}

\begin{document}
	
	\title{Complete mappings stabilizing Sylow $2$-subgroups in central quotients of special linear groups}
	\author[a]{Shixin Wang}
	\author[b]{Shikang Yu}
	\author[b,c]{Tao Feng \thanks{Supported by NSFC under Grant 12271023}}
	\affil[a]{School of Mathematical Sciences, Laboratory of Mathematics and Complex Systems, MOE, Beijing Normal University, 100875 Beijing, P.R. China}
	\affil[b]{School of Mathematics and Statistics, Beijing Jiaotong University, 100044 Beijing, P.R. China}
    \affil[c]{Hebei Provincial Key Laboratory of Mathematical Theory and Analysis for Network and Data Science, Beijing Jiaotong University, 100044 Beijing, P.R. China}
    \affil[ ]{sxwang@bjtu.edu.cn; healthyu@bjtu.edu.cn; tfeng@bjtu.edu.cn}
    
	\date{}	
	\renewcommand*{\Affilfont}{\small\it}
	\maketitle

\begin{abstract}
A complete mapping of a finite group $G$ is a permutation $\theta$ of $G$ such that $x\mapsto x\theta(x)$ is also a permutation of $G$. Complete mappings have applications in cryptography, particularly in the design of round functions for cryptographic systems. The study of their parity has led to the question of which finite groups admit complete mappings of both parities. A useful approach to this question is to construct complete mappings stabilizing suitable subgroups, in particular Sylow $2$-subgroups. We show that, for every $n\geq2$, every prime power $q$, and every $Z\leq Z(\mathrm{SL}_n(q))$, the group $\mathrm{SL}_n(q)/Z$ admits a complete mapping stabilizing a Sylow $2$-subgroup if and only if $(n,q)\neq(2,2)$.
\end{abstract}


\section{Introduction}

Let $G$ be a finite group. A permutation $\theta$ of $G$ is called a \emph{complete mapping} if $x\mapsto x\theta(x)$ is also a permutation of $G$. Equivalently, there exist a set $I$ and three bijections $A,B,C:I\longrightarrow G$ such that $A(i)B(i)=C(i)$ for all $i\in I$. The triple $(A,B,C)$, called a \emph{complete mapping triple} of $G$, defines the complete mapping $\theta=B\circ A^{-1}$, where $\circ$ denotes composition of maps, with the rightmost map applied first. Indeed, given a complete mapping $\theta$, one may take $I=G$, $A=\mathrm{id}_G$, $B=\theta$, and define $C:G\longrightarrow G$ by
$C(g)=g\theta(g)$ for every $g\in G$. Conversely, suppose that $(A,B,C)$ is a triple of bijections $I\longrightarrow G$ satisfying $A(i)B(i)=C(i)$ for all $i\in I$, and let $\theta=B\circ A^{-1}$. For $g\in G$, setting $i=A^{-1}(g)$ gives $g\theta(g)=A(i)(B\circ A^{-1})(g)=A(i)B(i)=C(i)=(C\circ A^{-1})(g)$. Since $C\circ A^{-1}$ is a bijection, the map $g\mapsto g\theta(g)$ is also a bijection. Thus $\theta$ is a complete mapping.

Complete mappings were introduced by Mann \cite{Mann} in order to construct orthogonal Latin squares. Hall and Paige \cite{HP} conjectured that a finite group admits a complete mapping if and only if its Sylow $2$-subgroups are trivial or noncyclic. This conjecture was eventually established through the work of Hall and Paige \cite{HP}, Wilcox \cite{Wilcox}, Evans \cite{Evans09}, and Bray et al. \cite{BCCSZ}.

\begin{thm}\rm{\cite{BCCSZ,Evans09,HP,Wilcox}}\label{thm:HP}
A finite group admits a complete mapping if and only if its	Sylow $2$-subgroups are trivial or noncyclic.
\end{thm}

A permutation $\phi$ of $G$ is called an \emph{orthomorphism} if $x\mapsto x^{-1}\phi(x)$ is also a permutation of $G$. For any complete mapping $\theta$ of $G$, the permutation $x\mapsto x\theta(x)$ is an orthomorphism, since $x^{-1}(x\theta(x))=\theta(x)$. We call this orthomorphism the
\emph{orthomorphism associated with} $\theta$. A permutation is called an \emph{odd} (resp. \emph{even}) permutation if it can be written as a product of an odd (resp. even) number of transpositions, where a transposition swaps two distinct elements. A complete mapping or orthomorphism of $G$ is said to be \emph{odd} (resp. \emph{even}) if it is an odd (resp. even) permutation of $G$.

Complete mappings and orthomorphisms have applications in cryptography, particularly in the design of round functions for cryptographic systems; see \cite{Loiss,FOX,Vaudenay} and the references therein. 
Motivated by these applications, Bors and Wang investigated when the permutation groups generated by all complete mappings or by all orthomorphisms of a finite group are ``large'' and primitive \cite[Question~1.1.1]{BW}. This investigation led them to study the parity of complete mappings and orthomorphisms \cite[Subsection~1.1]{BW}. They also observed that a finite group admits orthomorphisms of both parities if and only if it admits complete mappings of both parities. Bors and Wang further proved that complete mappings of both parities exist for every finite group of odd order greater than $3$, every finite solvable group of order greater than $24$ whose Sylow $2$-subgroups are trivial or noncyclic, and every noncyclic finite $2$-group except $C_2^2$, $C_2^3$, and the quaternion group $Q_8$ \cite[Theorem~1.1.6 and Proposition~4.2]{BW}. Here, for positive integers $m$, $C_m$ denotes the cyclic group of order $m$, and $C_2^m$ denotes the direct product of $m$ copies of $C_2$.

For a finite group $G$ and a subgroup $H\leq G$, a complete mapping $\theta$ of $G$ is said to \emph{stabilize} $H$ if $\theta(H)=H$. In this case, the restriction $\theta|_H$ is a complete mapping of $H$. Moreover, $\theta|_H$ can be replaced by any complete mapping of $H$, while leaving $\theta$ unchanged on $G\setminus H$, to obtain another complete mapping of $G$. Consequently, if $H$ admits complete mappings of both parities, then so does $G$ \cite[Lemma~1.1]{YFL}. Building on Wilcox's approach \cite{Wilcox}, Yu et al. \cite{YFL} reduced the problem of determining whether an arbitrary finite group admits complete mappings of both parities to constructing complete mappings stabilizing suitable subgroups, especially Sylow $2$-subgroups, in finite simple groups of Lie type and certain sporadic simple groups. For finite simple groups, we use the notation in \cite{Wilson}.

\begin{thm}\label{cor:main}\rm{\cite[Corollary~1.2]{YFL}}
If every finite simple group of Lie type and every sporadic simple group in $\{\mathrm{Ly}$, $\mathrm{Co}_2$, $\mathrm{B}$, $\mathrm{Th}$, $\mathrm{Fi}_{22}$, $\mathrm{Fi}_{23}$, $\mathrm{Fi}'_{24}$, $\mathrm{Ru}$, $\mathrm{O'N}$, $\mathrm{Co}_1$, $\mathrm{Co}_3$, $\mathrm{HN}$, $\mathrm{M}\}$ admits a complete mapping that stabilizes a Sylow $2$-subgroup, and every Ree group ${}^2G_2(3^{2n+1})$ with $n\geq1$ admits a complete mapping that stabilizes some subgroup of order
$8p_0^{t_0}$ for some odd prime $p_0$ and some integer $t_0\geq1$, then every finite nontrivial group $G$ with a trivial or noncyclic Sylow $2$-subgroup admits complete mappings of both parities, except when $G\cong C_3$, $C_2^2$, $C_2^3$, or $Q_8$.
\end{thm}

In this paper, we study complete mappings stabilizing Sylow $2$-subgroups in central quotients of special linear groups, namely groups of the form $\SL_n(q)/Z$, where $Z\leq Z(\SL_n(q))$. This includes $\SL_n(q)$ itself when $Z=1$ and $\PSL_n(q)$ when $Z=Z(\SL_n(q))$. In particular, our results establish the existence of complete mappings stabilizing Sylow $2$-subgroups for every simple projective special linear group, a classical family of finite simple groups of Lie type. Our main result is the following.

\begin{thm}\label{thm:main}
Let $n\geq2$, let $q$ be a prime power, and let $Z\leq Z(\SL_n(q))$.
Then $\SL_n(q)/Z$ admits a complete mapping stabilizing a Sylow $2$-subgroup if and only if $(n,q)\neq(2,2)$.
\end{thm}

The exception $(n,q)=(2,2)$ is necessary, since $\SL_2(2)\cong\PSL_2(2)\cong S_3$ has a nontrivial cyclic Sylow $2$-subgroup and hence admits no complete mapping by Theorem \ref{thm:HP}.
Taking $Z=1$ and $Z=Z(\SL_n(q))$ in Theorem~\ref{thm:main} yields the following corollary.

\begin{core}\label{cor:PSL}
Let $n\geq2$ and let $q$ be a prime power.
Then each of $\SL_n(q)$ and $\PSL_n(q)$ admits a complete mapping stabilizing a Sylow $2$-subgroup if and only if	$(n,q)\neq(2,2)$.
\end{core}

Our proof proceeds by constructing complete mappings on suitable subgroups and then extending them to the whole group. The key observation is that if $H\leq G$ has odd index and $P\in\mathrm{Syl}_2(H)$, then $P\in\mathrm{Syl}_2(G)$. Hence, any complete mapping of $G$ that extends a complete mapping of $H$ stabilizing $P$ also stabilizes $P$. Our constructions use the geometry of the vector space $V=\F_q^n$. For $G=\SL_n(q)/Z$, where $Z\leq Z(\SL_n(q))$, we employ two types of subgroups of $G$ arising from this geometry. The first is the stabilizer in $G$ of a suitable flag in $V$. The second is, when $V=U\oplus W\oplus F$ with $\dim U=\dim W>0$, the subgroup of $G$ preserving $F$ and the unordered pair $\{U,W\}$ setwise, where $F$ may be zero. When the characteristic of $\F_q$ is odd, both types of subgroups are needed at different stages of the proof. When the characteristic of $\F_q$ is even, the argument is simpler and only flag stabilizers are needed. More precisely, we use the image in $G$ of a Borel subgroup of $\SL_n(q)$, equivalently the stabilizer in $G$ of a complete flag in $V$.

The paper is organized as follows. Section~\ref{sec:preliminaries} collects the notation and preliminary results. Section~\ref{sec:general-lifting} develops the general extension tools, and Section~\ref{sec:linear-lifting} establishes the geometric extension results. Sections~\ref{sec:odd} and~\ref{sec:even} deal with the odd- and even-characteristic cases, respectively. Finally, Section~\ref{sec:8} completes the proof of the main theorem and concludes the paper with a discussion of the connection with difference matrices and possible extensions of our approach to other finite simple groups of Lie type and sporadic simple groups.

\section{Preliminaries}\label{sec:preliminaries}

In this section, we collect some definitions and preliminary results that will be used throughout the paper.

All groups are finite. For a subgroup $H\leq G$, we write $G/H=\{gH:g\in G\}$ for the set of left $H$-cosets; normality of $H$ is not assumed unless explicitly stated. We use $\mathbin{\dot\cup}$ to denote the disjoint union of sets.

When $G=\SL(V)/Z$ and $\pi:\SL(V)\to G$ is the natural projection, we use hats to denote objects in $\SL(V)$ and tildes to denote objects in $\GL(V)$. For a subgroup $\widehat H\leq\SL(V)$, we write $H=\pi(\widehat H)$ for its image in $G$. We write $D_{2d}$ for the dihedral group of order $2d$.

We shall also use the elementary fact that every finite group of odd order admits a complete mapping: the identity map is a complete mapping, since the squaring map $x\mapsto x^2$ is a permutation.

\subsection{Linear groups and flags}

In this subsection, we collect some facts about linear groups. For maximal tori, Borel subgroups, and Levi decompositions in $\SL_n(q)$, we refer to \cite[Chapters~11--12]{MalleTesterman}. For split and nonsplit tori, as well as the Bruhat decomposition of $\SL_2(q)$, we refer to \cite[Chapter~1]{Bonnafe2011}.

\begin{defi}\label{def:Borel}
In $\SL_n(q)$, let
\[
\widehat T^{+}:=\left\{\diag(t_1,\ldots,t_n):t_i\in\F_q^\times,\ t_1\cdots t_n=1\right\}\cong(\F_q^\times)^{n-1}\cong C_{q-1}^{\,n-1}.
\]
A subgroup of $\SL_n(q)$ is a \emph{maximally split torus} if and only if it is conjugate to $\widehat T^{+}$. Let
\[
\widehat B:=\{A\in\SL_n(q):A\text{ is upper triangular}\}.
\]
A subgroup of $\SL_n(q)$ is a \emph{Borel subgroup} if and only if it is conjugate to $\widehat B$. Let
\[
\widehat U:=\{A\in\widehat B:\text{all diagonal entries of }A\text{ are }1\}.
\]
The subgroup $\widehat U$ is called the \emph{unipotent radical} of $\widehat B$, and
\[
\widehat B=\widehat U\rtimes\widehat T^{+},
\]
with $\widehat T^{+}$ acting on $\widehat U$ by conjugation.
\end{defi}

\begin{defi}\label{def:nonsplit-torus}
Let $n=2$ and let $q$ be odd. Fix a nonsquare $\varepsilon\in\F_q^\times$. Let
\[
\widetilde T^{-}:=\left\{
\begin{pmatrix}
a&b\\ \varepsilon b&a
\end{pmatrix}
:a,b\in\F_q,\ (a,b)\neq(0,0)\right\}\cong\F_{q^2}^{\times}\cong C_{q^2-1}.
\]
A subgroup of $\GL_2(q)$ is a \emph{nonsplit maximal torus} if and only if it is conjugate to $\widetilde T^{-}$. Let
\[
\widehat T^{-}:=\widetilde T^{-}\cap\SL_2(q)=\left\{
\begin{pmatrix}
a&b\\ \varepsilon b&a
\end{pmatrix}
:a,b\in\F_q,\ a^2-\varepsilon b^2=1\right\}\cong C_{q+1}.
\]
A subgroup of $\SL_2(q)$ is a \emph{nonsplit maximal torus} if and only if it is conjugate to $\widehat T^{-}$.
\end{defi}

We now recall the definition of flags in a vector space.

\begin{defi}\label{def:flag}
A \emph{flag} in a vector space $V$ is a strictly increasing sequence of subspaces
\[
\mathcal F: 0=V_0<V_1<\cdots<V_r=V.
\]
If $\dim V=n$, $r=n$, and $\dim V_i=i$, $0\leq i\leq n$, then $\mathcal F$ is called a \emph{complete flag}. Otherwise, it is called a \emph{partial flag}.

If $X$ is a group acting on the subspaces of $V$, the stabilizer of $\mathcal F$ in $X$ is denoted by
\[
X_{\mathcal F}=\left\{x\in X:x(V_i)=V_i,\ 0\leq i\leq r\right\}.
\]
When $X=\SL_n(q)$, the stabilizer of a flag is called a \emph{parabolic subgroup} of $\SL_n(q)$. If the flag is complete, then its stabilizer is a Borel subgroup. Conversely, every parabolic subgroup of $\SL_n(q)$ arises as the stabilizer of a flag, and every Borel subgroup arises as the stabilizer of a complete flag, see, for example, \cite[Chapter~12]{MalleTesterman}.
\end{defi}

In Definition~\ref{def:Borel}, we recorded a decomposition of a Borel subgroup of $\SL_n(q)$. Since Borel subgroups arise as the stabilizers of complete flags, we now give the corresponding decomposition for the stabilizer of an arbitrary flag, namely, for a parabolic subgroup.

\begin{lem}\label{lem:flag-Levi}{\rm \cite[Section 2.5, Proposition 2]{AlperinBell1995}}
Let $\mathcal F:0=V_0<V_1<\cdots<V_r=V=\F_q^n$ be a flag, and let $\widehat P:=\SL(V)_{\mathcal F}$. For each $1\leq i\leq r$, choose a complement $E_i$ of $V_{i-1}$ in $V_i$. Then $V_i=V_{i-1}\oplus E_i$, $V=E_1\oplus\cdots\oplus E_r$. Put $d_i=\dim E_i$. With respect to this direct sum decomposition, define
\[
\widehat U_{\mathcal F}
=\left\{\begin{pmatrix} I_{d_1}&*&\cdots&*\\ 0&I_{d_2}&\cdots&*\\ \vdots&&\ddots&\vdots\\ 0&\cdots&0&I_{d_r}\end{pmatrix}\right\},
\]
where each $*$ denotes an arbitrary matrix of the appropriate size, and
\[
\widehat L_{\mathcal F}=\left\{\diag(A_1,\ldots,A_r): A_i\in\GL_{d_i}(q),\ \prod_{i=1}^{r}\det A_i=1\right\}.
\]
Then $\widehat P=\widehat U_{\mathcal F}\rtimes\widehat L_{\mathcal F}$, where $\widehat L_{\mathcal F}$ acts on $\widehat U_{\mathcal F}$ by conjugation. Here $\widehat U_{\mathcal F}$ is the unipotent radical of $\widehat P$, and $\widehat L_{\mathcal F}$ is called a Levi subgroup of $\widehat P$.
\end{lem}

We shall frequently pass between a flag and a refinement of it.

\begin{defi}\label{def:flag-refinement}
Let $\mathcal F:0=V_0<V_1<\cdots<V_r=V$ be a flag. A flag $\mathcal F'$ is called a \emph{refinement} of $\mathcal F$ if every member of $\mathcal F$ is also a member of $\mathcal F'$. We call $\mathcal F'$ the \emph{finer flag} and $\mathcal F$ the \emph{coarser flag}.
\end{defi}

In particular, inserting a subspace $\widetilde U$ with $V_{i-1}<\widetilde U<V_i$ gives the refinement
\[
\mathcal F': 0=V_0<\cdots<V_{i-1}<\widetilde U<V_i<\cdots<V_r=V.
\]
In this single-step refinement, if $\dim(V_i/V_{i-1})=a+b$ and $\dim(\widetilde U/V_{i-1})=a$, then $\dim(V_i/\widetilde U)=b$. Thus inserting $\widetilde U$ amounts to choosing an $a$-dimensional subspace $\widetilde U/V_{i-1}\leq V_i/V_{i-1}$.

Finally, we record the Bruhat decomposition of $\SL_2(q)$.

\begin{lem}\label{lem:SL2-Bruhat}{\rm \cite[Section 1.1.1]{Bonnafe2011}}
Let $\widehat B=\widehat T^{+}\widehat U$ be a Borel subgroup of $\SL_2(q)$, and let
\[
\widehat w=
\begin{pmatrix}
0&-1\\1&0
\end{pmatrix}.
\]
Then $\SL_2(q)=\widehat B\mathbin{\dot\cup}\widehat B\widehat w\widehat B=\widehat T^{+}\widehat U\mathbin{\dot\cup}\widehat U\widehat w\widehat T^{+}\widehat U$.
\end{lem}

\subsection{Gaussian binomial coefficients}

We conclude this section with two elementary facts about Gaussian binomial coefficients that will be used later.

\begin{lem}\label{lem:qbinom-parity}
For integers $0\leq s\leq r$, the Gaussian binomial coefficient is defined by $\binom{r}{s}_q:=\prod\limits_{i=0}^{s-1}\frac{q^{r-i}-1}{q^{s-i}-1}$. Then the following hold.
\begin{enumerate}
\item[$(1)$] If $q$ is even, then $\binom{r}{s}_q$ is odd.
\item[$(2)$] If $q$ is odd, then $\binom{r}{s}_q\equiv\binom{r}{s}\pmod 2$. In particular, if the binary addition of $s$ and $r-s$ involves no carries, then $\binom{r}{s}_q$ is odd.
\end{enumerate}
\end{lem}

\begin{proof}
The Gaussian binomial coefficients satisfy the recurrence $\binom{r}{s}_q=\binom{r-1}{s}_q+q^{\,r-s}\binom{r-1}{s-1}_q$, with $\binom{r}{0}_q=\binom{r}{r}_q=1$, see \cite[Theorem~3.2]{Andrews}.

Suppose first that $q$ is even. If $s=0$ or $s=r$, then $\binom{r}{s}_q=1$. Now assume that $0<s<r$. Since $q\equiv0\pmod2$, the recurrence gives $\binom{r}{s}_q\equiv\binom{r-1}{s}_q\pmod2$. Iterating this congruence yields
\[
\binom{r}{s}_q\equiv\binom{r-1}{s}_q\equiv\cdots\equiv\binom{s}{s}_q=1\pmod2.
\]
Thus $\binom{r}{s}_q$ is odd.

Now suppose that $q$ is odd. We prove $\binom{r}{s}_q\equiv\binom{r}{s}\pmod2$ simultaneously for all $0\leq s\leq r$ by induction on $r$. The case $r=0$ is immediate. Assume that $r\geq1$ and that the assertion holds for $r-1$. If $s=0$ or $s=r$, the assertion is immediate. Now let $0<s<r$. By the induction hypothesis, $\binom{r-1}{s}_q\equiv\binom{r-1}{s}\pmod2$, $\binom{r-1}{s-1}_q\equiv\binom{r-1}{s-1}\pmod2$. Since $q\equiv1\pmod2$, the recurrence gives
\[
\binom{r}{s}_q\equiv\binom{r-1}{s}_q+\binom{r-1}{s-1}_q\equiv\binom{r-1}{s}+\binom{r-1}{s-1}=\binom{r}{s}\pmod2.
\]

Finally, suppose that the binary addition of $s$ and $r-s$ involves no carries. Write $s=\sum_{i\geq0}s_i2^i$, $r-s=\sum_{i\geq0}t_i2^i$, where $s_i,t_i\in\{0,1\}$. Then $s_i+t_i\leq1$ for every $i$, and hence $r=\sum_{i\geq0}(s_i+t_i)2^i$. By Lucas' theorem,
\[
\binom{r}{s}\equiv\prod_{i\geq0}\binom{s_i+t_i}{s_i}\equiv1\pmod2.
\]
By part~$(2)$, $\binom{r}{s}_q\equiv\binom{r}{s}\equiv1\pmod2$. Therefore $\binom{r}{s}_q$ is odd.
\end{proof}

\begin{lem}\label{lem:equal-qbinom-valuation}
For a positive integer $n$, let $\nu_2(n):=\max\{e\geq0:2^e\mid n\}$ denote the $2$-adic valuation of $n$. Suppose that $q$ is odd and $m$ is a power of $2$. Then
\begin{itemize}
\item[$(1)$] $\nu_2\left(\binom{2}{1}_q\right)=\nu_2(q+1)$;
\item[$(2)$] $\nu_2\left(\binom{2m}{m}_q\right)=1$ if $m\geq 2$.
\end{itemize}
\end{lem}

\begin{proof}
Since $\binom{2}{1}_q=q+1$, we have $(1)$.

Assume that $m\geq2$. Since $m$ is a power of $2$, $m$ is even. Moreover, $\binom{2m}{m}_q=\prod\limits_{i=1}^{m}\frac{q^{m+i}-1}{q^i-1}$. Hence
\[
\nu_2\left(\binom{2m}{m}_q\right)=\sum_{i=1}^{m}\left(\nu_2(q^{m+i}-1)-\nu_2(q^i-1)\right).
\]

If $i$ is odd, then both $i$ and $m+i$ are odd. Since $\nu_2(q^n-1)=\nu_2(q-1)$ for every odd positive integer $n$, we have $\nu_2(q^{m+i}-1)-\nu_2(q^i-1)=0$.

Now suppose that $i$ is even. By the lifting-the-exponent lemma \cite[Theorem~6.26]{AndreescuDospinescuMushkarov}, $\nu_2(q^{m+i}-1)-\nu_2(q^i-1)=\nu_2(m+i)-\nu_2(i)$. Since $m$ is a power of $2$, for every $i<m$ we have $\nu_2(m+i)=\nu_2(i)$, whereas for $i=m$, $\nu_2(2m)-\nu_2(m)=1$. Therefore $\nu_2\left(\binom{2m}{m}_q\right)=1$.
\end{proof}

\section{Lifting constructions in general groups}\label{sec:general-lifting}

In this section, we develop several general constructions for lifting complete mappings. These constructions will be used repeatedly in the subsequent sections.

The following criterion, due to Wilcox \cite{Wilcox}, allows one to construct a complete mapping from suitable data on cosets.

\begin{thm}{\rm \cite[Proposition~14]{Wilcox}}\label{thm:wilcox}
Let $H\leq G$, and suppose that $H$ admits a complete mapping. Let $I$ be an index set with $|I|=[G:H]$. Suppose that there exist three bijections
\[
\widetilde x,\widetilde y,\widetilde z:I\longrightarrow G/H
\]
such that $\widetilde z(i)\subseteq\widetilde x(i)\widetilde y(i)$ for every $i\in I$. Then $G$ admits a complete mapping.
\end{thm}

We first give a lifting construction for normal extensions.

\begin{lem}\label{lem:normal-extension}
Let $N\trianglelefteq G$, let $P\in\mathrm{Syl}_2(G)$, and let $P_0=P\cap N$, $\overline P=PN/N\in\mathrm{Syl}_2(G/N)$. If $N$ has a complete mapping stabilizing $P_0$, and $G/N$ has a complete mapping stabilizing $\overline P$, then $G$ has a complete mapping stabilizing $P$.
\end{lem}

\begin{proof}
Let $\overline G=G/N$, and let $P_0=P\cap N$, $\overline P=PN/N$. Let $\varphi:N\longrightarrow N$ be a complete mapping stabilizing $P_0$, and let $\overline\theta:\overline G\longrightarrow\overline G$ be a complete mapping stabilizing $\overline P$.

For each $t\in\overline G$, choose $x_t,y_t\in G$ such that
\[
x_tN=t,\quad y_tN=\overline\theta(t).
\]
If $t\in\overline P=PN/N$, then $t$ has a representative in $P$. Moreover, $\overline\theta(t)\in\overline P$, because $\overline\theta(\overline P)=\overline P$. Hence we may choose $x_t,y_t\in P$.

Define $A,B,C:\overline G\times N\longrightarrow G$ by
\[
A(t,n)=x_ty_tny_t^{-1},\quad B(t,n)=y_t\varphi(n),\quad C(t,n)=x_ty_tn\varphi(n).
\]
Then
\[
A(t,n)B(t,n)=x_ty_tny_t^{-1}y_t\varphi(n)=x_ty_tn\varphi(n)=C(t,n).
\]

For fixed $t\in\overline G$, since $N\trianglelefteq G$, the map $n\mapsto y_tny_t^{-1}$ is a permutation of $N$. Hence, as $n$ runs through $N$, the values of $A(t,n)$ run through $x_tN$. Since $\varphi$ is a permutation of $N$, the values of $B(t,n)=y_t\varphi(n)$ run through $y_tN$. Since $\varphi$ is a complete mapping of $N$, the map $n\mapsto n\varphi(n)$ is also a permutation of $N$. Hence the values of $C(t,n)=x_ty_tn\varphi(n)$ run through $x_ty_tN$. Moreover, $x_tN=t$, $y_tN=\overline\theta(t)$, $x_ty_tN=t\overline\theta(t)$. Since
\[
t\mapsto t,\quad t\mapsto\overline\theta(t),\quad t\mapsto t\overline\theta(t)
\]
are permutations of $\overline G$, it follows that $A,B,C$ are bijections from $\overline G\times N$ onto $G$. Thus $(A,B,C)$ is a complete mapping triple of $G$. Let $\theta=B\circ A^{-1}$ be the corresponding complete mapping.

It remains to show that $\theta(P)=P$. Let $t\in\overline P$ and $n\in P_0$. By our choice of representatives, $x_t,y_t\in P$. Since $P_0\trianglelefteq P$, we have $y_tny_t^{-1}\in P_0$. Since $\varphi(P_0)=P_0$, we also have $\varphi(n)$, $n\varphi(n)\in P_0$. Hence
\[
A(t,n),\ B(t,n),\ C(t,n)\in P.
\]
Therefore the restrictions of $A,B,C$ to $\overline P\times P_0$ take values in $P$. Since $A,B,C$ are bijections on $\overline G\times N$, these restrictions are injective. Moreover, since $P/P_0\cong\overline P$, $|\overline P\times P_0|=|\overline P|\,|P_0|=|P|$. Hence each restriction is a bijection from $\overline P\times P_0$ onto $P$. Consequently, $\theta(P)=P$.
\end{proof}

For a subset $S\subseteq G$, a \emph{complete mapping triple on $S$} is a triple of bijections $A,B,C:J\longrightarrow S$ satisfying $A(j)B(j)=C(j)$. The following lemma is a local version of Theorem~\ref{thm:wilcox} with prescribed representatives.

\begin{lem}\label{lem:local-wilcox}
Let $H\leq G$, let $a,b,c:J\to H$ be a complete mapping triple, and let $S\subseteq G$ be a union of both left and right $H$-cosets. Set $S/H=\{sH:s\in S\}$. Suppose that
\[
\widetilde x,\widetilde y,\widetilde z:I\longrightarrow S/H
\]
are bijections satisfying $\widetilde z(i)\subseteq\widetilde x(i)\widetilde y(i)$, $i\in I$, and that elements $y_i\in\widetilde y(i)$ can be chosen so that $\{y_i:i\in I\}$ is simultaneously a transversal for the left and right $H$-cosets in $S$. Then there exists a complete mapping triple $(A,B,C)$ on $S$ with common index set $I\times J$.

If, in addition, there exists $j_0\in J$ such that $a(j_0)=b(j_0)=c(j_0)=1$, and $x_i\in\widetilde x(i)$, $z_i\in\widetilde z(i)$ are prescribed for some indices $i\in I$ with $z_i=x_iy_i$, then the triple may be chosen so that $A(i,j_0)=x_i$, $B(i,j_0)=y_i$, $C(i,j_0)=z_i$ for every such $i$.
\end{lem}

\begin{proof}
For each $i\in I$ for which $x_i$ and $z_i$ have not been prescribed, choose $w_i\in\widetilde z(i)$. Since $\widetilde z(i)\subseteq\widetilde x(i)\widetilde y(i)$, there exist $x_i\in\widetilde x(i)$ and $u_i\in\widetilde y(i)$ such that $w_i=x_iu_i$. Since $u_i,y_i\in\widetilde y(i)$, we have $u_iH=y_iH=\widetilde y(i)$, and hence $x_i\widetilde y(i)=x_iu_iH=w_iH=\widetilde z(i)$. Set $z_i:=x_iy_i$. Then $z_iH=x_iy_iH=x_i\widetilde y(i)=\widetilde z(i)$. For an index $i$ for which $x_i$ and $z_i$ are prescribed, we use the prescribed representatives. Since $z_i=x_iy_i$, we again have $x_i\widetilde y(i)=x_iy_iH=z_iH=\widetilde z(i)$. Thus, for every $i\in I$,
\[
x_iH=\widetilde x(i),\quad z_iH=\widetilde z(i).
\]
Since $\widetilde x,\widetilde z:I\longrightarrow S/H$ are bijections, the sets $\{x_i:i\in I\}$ and $\{z_i:i\in I\}$ are transversals for the left $H$-cosets contained in $S$.

For each $(i,j)\in I\times J$, since $\{y_i:i\in I\}$ is a transversal for the right $H$-cosets contained in $S$, there exist unique $d(i,j)\in H$, $r(i,j)\in I$ such that $y_ia(j)=d(i,j)y_{r(i,j)}$. Define
\begin{align*}
A(i,j)= x_i d(i,j),\quad B(i,j)= y_{r(i,j)}b(j),\quad C(i,j)= z_i c(j).
\end{align*}
Then $A(i,j)B(i,j)= x_i d(i,j)y_{r(i,j)}b(j)= x_i y_i a(j)b(j)= z_i c(j)= C(i,j)$.

We now show that $A,B,C$ are bijections. Since $|I\times J| = |S/H| |H| = |S|$, it suffices to prove injectivity. Suppose first that $A(i,j)=A(i',j')$. Since $\{x_i:i\in I\}$ is a transversal for the left $H$-cosets in $S$, we obtain $i=i'$, $d(i,j)=d(i,j')$. From
\[
y_{r(i,j)}a(j)^{-1} = d(i,j)^{-1}y_i = d(i,j')^{-1}y_i = y_{r(i,j')}a(j')^{-1},
\]
and the fact that $\{y_i:i\in I\}$ is also a transversal for the left $H$-cosets, we obtain $r(i,j)=r(i,j')$ and then $a(j)=a(j')$. Since $a$ is injective, $j=j'$. Thus $A$ is injective. Next suppose that $B(i,j)=B(i',j')$. Using the left $H$-coset transversal property of $\{y_i\}$, we obtain $r(i,j)=r(i',j')$, $b(j)=b(j')$. Since $b$ is injective, $j=j'$, and so $a(j)=a(j')$. Consequently,
\[
d(i,j)^{-1}y_i = y_{r(i,j)}a(j)^{-1} = y_{r(i',j')}a(j')^{-1} = d(i',j')^{-1}y_{i'}.
\]
The two sides belong respectively to the right cosets $Hy_i$ and $Hy_{i'}$. Since $\{y_i:i\in I\}$ is a transversal for the right $H$-cosets, it follows that $i=i'$. Thus $B$ is injective. Finally, if $C(i,j)=C(i',j')$, then the left $H$-coset transversal property of $\{z_i\}$ gives $i=i'$, $c(j)=c(j')$. Since $c$ is injective, $j=j'$. Thus $C$ is injective. Hence $A,B,C$ are bijections and form a complete mapping triple on $S$.

Finally, suppose that there exists $j_0\in J$ such that $a(j_0)=b(j_0)=c(j_0)=1$. Then $y_i a(j_0)=y_i$, so the uniqueness of the decomposition $y_i a(j)=d(i,j)y_{r(i,j)}$ gives $d(i,j_0)=1$, $r(i,j_0)=i$. Therefore $A(i,j_0)=x_i$, $B(i,j_0)=y_i$, $C(i,j_0)=z_i$.
\end{proof}

Suppose that $G=KH$ is an exact factorization, that is, $G=KH$, $K\cap H=1$. Lemma \ref {lem:local-wilcox} allows us to extend complete mappings along an exact factorization while preserving prescribed values on one factor.

\begin{lem}\label{lem:exact-factorization}
Suppose that $G=KH$ is an exact factorization. If $K$ admits a complete mapping and $\alpha$ is a complete mapping of $H$, then there exists a complete mapping $\theta$ of $G$ such that $\theta|_H=\alpha$. In particular, if $P\leq H$ and $\alpha(P)=P$, then $\theta(P)=P$.
\end{lem}

\begin{proof}
Let $\alpha$ be the prescribed complete mapping of $H$, and let $\beta$ be a complete mapping of $K$. Define $\beta'(k)=\beta(k)\beta(1_K)^{-1}$. It is straightforward to verify that $\beta'$ is a complete mapping of $K$, and that $\beta'(1_K)=1_K$. Hence, without loss of generality, we may assume that $\beta(1_K)=1_K$. Then
\[
a(k)=k,\quad b(k)=\beta(k),\quad c(k)=k\beta(k)
\]
define a complete mapping triple of $K$, with $a(1_K)=b(1_K)=c(1_K)=1_K$.

Since $G=KH$ is an exact factorization, every element of $G$ has a unique expression $kh$ with $k\in K$ and $h\in H$. Hence $H$ is a transversal for the right $K$-cosets of $G$. Similarly, $G=HK$ is an exact factorization, so $H$ is also a transversal for the left $K$-cosets. Define maps $\widetilde x,\widetilde y,\widetilde z:H\longrightarrow G/K$ by
\[
\widetilde x(h)=hK,\quad\widetilde y(h)=\alpha(h)K,\quad\widetilde z(h)=h\alpha(h)K .
\]
Since $H$ is a transversal for the left $K$-cosets of $G$, the map $H\longrightarrow G/K$, $h\mapsto hK$ is a bijection. Moreover, the maps $h\mapsto h$, $h\mapsto\alpha(h)$, $h\mapsto h\alpha(h)$ are permutations of $H$. Hence $\widetilde x,\widetilde y,\widetilde z$ are bijections from $H$ onto $G/K$. Furthermore, $h\alpha(h)K\subseteq (hK)(\alpha(h)K)$, and therefore $\widetilde z(h)\subseteq\widetilde x(h)\widetilde y(h)$ for every $h\in H$.

Choose the representatives $x_h=h\in\widetilde x(h)$, $y_h=\alpha(h)\in\widetilde y(h)$, $z_h=h\alpha(h)\in\widetilde z(h)$. Then $z_h=x_hy_h$. Since $a(1_K)=b(1_K)=c(1_K)=1_K$, the additional assertion of Lemma~\ref{lem:local-wilcox}, with $J=K$ and $j_0=1_K$, yields a complete mapping triple $(A,B,C)$ of $G$ such that
\[
A(h,1_K)=x_h,\quad B(h,1_K)=y_h,\quad C(h,1_K)=z_h,\quad h\in H.
\]
Let $\theta=B\circ A^{-1}$ be the corresponding complete mapping of $G$. Since $A(h,1_K)=h$, $A^{-1}(h)=(h,1_K)$. Therefore, $\theta(h)=B(h,1_K)=\alpha(h)$ for every $h\in H$. Hence $\theta|_H=\alpha$. In particular, if $P\leq H$ and $\alpha(P)=P$, then $\theta(P)=\alpha(P)=P$.
\end{proof}

We next record a useful consequence of Lemma~\ref{lem:local-wilcox} for a single double coset.

\begin{lem}\label{lem:single-double-coset}
Let $H\leq G$ admit a complete mapping, and let $D=HgH$ be an $H$-double coset. If $D\subseteq D^2$, then $D$ admits a complete mapping triple.
\end{lem}

\begin{proof}
Since $D\subseteq D^2$, we may choose $x,y,z\in D$ such that $z=xy$. Put
\[
H_x:=H\cap xHx^{-1},\quad H_y:=H\cap yHy^{-1},\quad H_z:=H\cap zHz^{-1}.
\]
Since $x,y,z$ lie in the same $H$-double coset $D$, the subgroups $H_x,H_y,H_z$ are conjugate in $H$, and hence $[H:H_x]=[H:H_y]=[H:H_z]$. Indeed, if $y=h_1xh_2$ with $h_1,h_2\in H$, then $H_y=H\cap yHy^{-1}=h_1(H\cap xHx^{-1})h_1^{-1}=h_1H_xh_1^{-1}$.

By \cite[Lemma~13]{Wilcox}, there exists a subset $T\subseteq H$ which is simultaneously a left transversal for $H_x$ and $H_z$ in $H$. Choose a left transversal $T_y\subseteq H$ for $H_y$ in $H$. Since $[H:H_x]=[H:H_y]=[H:H_z]$, we may fix a bijection $\mu:T\longrightarrow T_y$. For each $t\in T$, define
\[
\widetilde x(t):=txH,\quad \widetilde y(t):=\mu(t)yH,\quad \widetilde z(t):=tzH.
\]
Since $T$ is a left transversal for both $H_x$ and $H_z$ in $H$, the maps $t\mapsto txH$ and $t\mapsto tzH$ are bijections from $T$ onto the set of left $H$-cosets contained in $D$. Similarly, since $T_y$ is a left transversal for $H_y$ in $H$ and $\mu:T\to T_y$ is a bijection, the map $t\mapsto \mu(t)yH$ is also a bijection onto the same set. Hence $\widetilde x,\widetilde y,\widetilde z$ are bijections from $T$ onto the set of left $H$-cosets contained in $D$.

We next show that $\widetilde z(t)\subseteq\widetilde x(t)\widetilde y(t)$. Since $z=xy$, for every $t\in T$,
\[
tz=txy=(tx\mu(t)^{-1})(\mu(t)y).
\]
Here $tx\mu(t)^{-1}\in txH=\widetilde x(t)$ and $\mu(t)y\in \mu(t)yH=\widetilde y(t)$. Hence $tz\in \widetilde x(t)\widetilde y(t)$. It follows that $\widetilde z(t)=tzH\subseteq\widetilde x(t)\widetilde y(t)$.

Next we choose representatives $y_t\in\widetilde y(t)$, $t\in T$, which simultaneously form a transversal for the left and right $H$-cosets contained in $D$. Let $\mathcal L$ and $\mathcal R$ denote, respectively, the sets of left and right $H$-cosets contained in $D$. Since $|D|=|\mathcal L||H|=|\mathcal R||H|$, we have $|\mathcal L|=|\mathcal R|$. Moreover, every left $H$-coset contained in $D$ has the form $agH$ for some $a\in H$, and every right $H$-coset contained in $D$ has the form $Hgb$ for some $b\in H$. Hence $agb\in agH\cap Hgb$. Thus every member of $\mathcal L$ has nonempty intersection with every member of $\mathcal R$.

Since $\widetilde y:T\longrightarrow\mathcal L$ is a bijection and $|\mathcal L|=|\mathcal R|$, choose a bijection $t\mapsto R_t$ from $T$ onto $\mathcal R$. By the preceding observation, $\widetilde y(t)\cap R_t\neq\varnothing$ for every $t\in T$. Thus for each $t\in T$, we can choose $y_t\in \widetilde y(t)\cap R_t$. Then $\{y_t:t\in T\}$ is simultaneously a transversal for the left and right $H$-cosets contained in $D$. Thus all the hypotheses of Lemma~\ref{lem:local-wilcox} are satisfied, and hence $D$ admits a complete mapping triple.
\end{proof}

The following observation allows us to patch together complete mapping triples on disjoint subsets.

\begin{lem}\label{lem:patching}
Let $S$ be a finite subset of a group $G$ with a disjoint decomposition $S=S_1\mathbin{\dot\cup}\cdots\mathbin{\dot\cup}S_t$. If each $S_i$ admits a complete mapping triple, then $S$ admits a complete mapping triple.
\end{lem}

\begin{proof}
For each $i$, let $(A_i,B_i,C_i)$ be a complete mapping triple of $S_i$, indexed by $I_i$. Replacing the $I_i$ by disjoint copies if necessary, set $I=I_1\mathbin{\dot\cup}\cdots\mathbin{\dot\cup}I_t$, and define $A,B,C:I\longrightarrow S$ by taking the disjoint unions of the maps $A_i,B_i,C_i$, respectively. Since the sets $S_i$ are pairwise disjoint, $A,B,C$ are bijections, and $A(u)B(u)=C(u)$ holds on each $I_i$. Hence $(A,B,C)$ is a complete mapping triple of $S$.
\end{proof}

\section{Lifting constructions in linear groups}\label{sec:linear-lifting}

In this section, we apply the general constructions of the preceding section to subgroups arising from decompositions and flags of a vector space.

We first establish a lifting construction between two flag stabilizers that differ by a single refinement. We then establish a second lifting construction for the case where $V$ is decomposed into subspaces, two of which have the same dimension.

Throughout this section, let $q$ be a prime power and let $V$ be a finite-dimensional $\F_q$-vector space. Let
\[
Z\leq Z(\SL(V)),\quad G:=\SL(V)/Z,
\]
and let
\[
\pi:\SL(V)\longrightarrow G
\]
be the natural quotient map. Since every element of $Z$ is a scalar transformation and hence fixes every subspace of $V$, the natural action of $\SL(V)$ on the subspaces of $V$ induces an action of $G$.

\subsection{Lifting along flags}\label{subsec:flag-lifting}

We begin with a simple geometric observation about subspaces.

\begin{lem}\label{lem:grassmann-triangle}
Let $E,F\leq V$ be $k$-dimensional subspaces such that $\dim(E\cap F)=k-j$. Then there exists a $k$-dimensional subspace $W\leq V$ such that $\dim(E\cap W)=\dim(W\cap F)=k-j$.
\end{lem}

\begin{proof}
Put $C=E\cap F$, so that $\dim C=k-j$. Choose complements $E_1$ and $F_1$ of $C$ in $E$ and $F$, respectively, so that $E=C\oplus E_1$, $F=C\oplus F_1$. Then $\dim E_1=\dim F_1=j$, $E+F=C\oplus E_1\oplus F_1$. Choose an isomorphism $f:E_1\to F_1$, and put
\[
\Gamma_f:=\{x+f(x):x\in E_1\}\subseteq E_1\oplus F_1.
\]
Since $E_1\cap F_1=0$, the map $E_1\longrightarrow E_1\oplus F_1$, $x\mapsto x+f(x)$, is linear, and it is injective since $E_1\cap F_1=0$. Hence $\Gamma_f$ is a $j$-dimensional subspace of $E_1\oplus F_1$. Moreover, by $E+F=C\oplus E_1\oplus F_1$, we have $C\cap\Gamma_f=0$. Set $W:=C\oplus \Gamma_f$. Then $\dim W=(k-j)+j=k$.

Since $C\leq E$ and $C\leq W$, we have $C\leq E\cap W$. Conversely, let $v\in E\cap W$. Since $E=C\oplus E_1$ and $W=C\oplus \Gamma_f$, we can write
\[
v=c_1+u=c_2+x+f(x),
\]
where $c_1,c_2\in C$ and $u,x\in E_1$. Comparing the $F_1$-components in the direct sum $E+F=C\oplus E_1\oplus F_1$ gives $f(x)=0$, and hence $x=0$. Thus $v=c_2\in C$, so $E\cap W=C$. Similarly, by comparing the $E_1$-components, we obtain $W\cap F=C$. Therefore $\dim(E\cap W)=\dim(W\cap F)=k-j$.
\end{proof}

We next consider a refinement of a flag obtained by inserting a single subspace between two consecutive members $V_{i-1}<V_i$.

Let $\mathcal F:0=V_0<V_1<\cdots<V_r=V$ be a flag. Fix $1\leq i\leq r$, and choose a subspace $\widetilde U$ with $V_{i-1}<\widetilde U<V_i$. Let $\mathcal F'$ be the refinement of $\mathcal F$ obtained by inserting $\widetilde U$. Set
\[
X:=G_{\mathcal F},\quad H:=G_{\mathcal F'}.
\]

To compare $H$ with $X$, we pass to the quotient space $L:=V_i/V_{i-1}$. Write $\dim L=a+b$, and assume that $\dim(\widetilde U/V_{i-1})=a$. Put $U:=\widetilde U/V_{i-1}\leq L$. Let
\[
\operatorname{Gr}_a(L):=\{Y\leq L:\dim Y=a\}
\]
denote the set of all $a$-dimensional subspaces of $L$. We classify these subspaces according to their intersection dimension with $U$. For $0\leq j\leq\min\{a,b\}$, define
\[
\Omega_j:=\bigl\{Y\in\operatorname{Gr}_a(L):\dim(U\cap Y)=a-j\bigr\},\quad D_j:=\bigl\{x\in X:\dim(U\cap xU)=a-j\bigr\}.
\]

\begin{lem}\label{lem:adjacent-quotient-transitivity}
The group $X$ acts transitively on $\operatorname{Gr}_a(L)$. Moreover, $H=X_U:=\{x\in X:xU=U\}$, and $[X:H]=\binom{a+b}{a}_q$.
\end{lem}

\begin{proof}
For each $1\leq s\leq r$, choose a complement $M_s$ of $V_{s-1}$ in $V_s$. Then $V_s=M_1\oplus\cdots\oplus M_s$, $1\leq s\leq r$, and hence $V=M_1\oplus\cdots\oplus M_r$. Moreover, the map
\[
M_s\longrightarrow V_s/V_{s-1},\quad m\mapsto m+V_{s-1},
\]
is a linear isomorphism. In particular, $M_i\cong L=V_i/V_{i-1}$.

We first note that $\SL(L)$ acts transitively on $\operatorname{Gr}_a(L)$. Indeed, let $Y_1,Y_2\in\operatorname{Gr}_a(L)$. There exists $s_0\in\GL(L)$ such that $s_0(Y_1)=Y_2$. Write $\det(s_0)=\lambda\in \mathbb{F}_q^\times$. Choose a basis $v_1,\ldots,v_{a+b}$ of $L$ such that $Y_2=\langle v_1,\ldots,v_a\rangle$. Define $t\in\GL(L)$ by
\[
t(v_1)=\lambda^{-1}v_1,\quad t(v_i)=v_i,\quad 2\leq i\leq a+b.
\]
Then $t(Y_2)=Y_2$ and $\det(t)=\lambda^{-1}$. Hence $s=ts_0\in\SL(L)$ and $s(Y_1)=Y_2$. Thus $\SL(L)$ acts transitively on $\operatorname{Gr}_a(L)$.

Via $M_i\cong L$, regard $s$ as a linear transformation of $M_i$, and define
\[
\widehat s=\mathrm{id}_{M_1}\oplus\cdots\oplus\mathrm{id}_{M_{i-1}}\oplus s\oplus\mathrm{id}_{M_{i+1}}\oplus\cdots\oplus\mathrm{id}_{M_r}.
\]
Then $\widehat s\in\SL(V)$, and $\widehat s$ stabilizes every member of $\mathcal F$. Hence $\pi(\widehat s)\in X$. The lift $\widehat s$ induces $s$ on $L$, and therefore $\pi(\widehat s)$ sends each subspace $Y\leq L$ to $s(Y)$. Since $\SL(L)$ acts transitively on $\operatorname{Gr}_a(L)$, it follows that $X$ acts transitively on $\operatorname{Gr}_a(L)$.

Since $\mathcal F'$ differs from $\mathcal F$ only by the additional subspace $\widetilde U$, $H=\{x\in X:x\widetilde U=\widetilde U\}$. As every $x\in X$ stabilizes $V_{i-1}$, $xU=U\Longleftrightarrow x\widetilde U=\widetilde U$. Therefore $H=X_U$. The orbit--stabilizer theorem now gives
\[
[X:H]=\bigl|\operatorname{Gr}_a(L)\bigr|=\binom{a+b}{a}_q.\qedhere
\]
\end{proof}

\begin{lem}\label{lem:intersection-double-cosets}
With the notation above, for each $0\leq j\leq\min\{a,b\}$, the group $H$ acts transitively on $\Omega_j$, and $D_j$ is an $H$-double coset. Moreover,
\[
X=D_0\mathbin{\dot\cup}D_1\mathbin{\dot\cup}\cdots\mathbin{\dot\cup}D_{\min\{a,b\}},\quad D_0=H,\quad D_j\subseteq D_j^2.
\]
\end{lem}

\begin{proof}
We first prove that $H$ is transitive on $\Omega_j$. Let $Y_1,Y_2\in\Omega_j$, and put $C_\ell=U\cap Y_\ell$, $\ell=1,2$. Then $\dim C_1=\dim C_2=a-j$. Choose decompositions $U=C_\ell\oplus A_\ell$, $Y_\ell=C_\ell\oplus B_\ell$, $\ell=1,2$. Thus $\dim A_1=\dim A_2=\dim B_1=\dim B_2=j$. Choose complements $R_\ell$ such that
\[
L=C_\ell\oplus A_\ell\oplus B_\ell\oplus R_\ell.
\]
Since the corresponding summands have equal dimensions, there are linear isomorphisms
\[
\alpha:C_1\to C_2,\quad \beta:A_1\to A_2,\quad \gamma:B_1\to B_2,\quad \delta:R_1\to R_2.
\]
Hence
\[
\sigma_0:=\alpha\oplus\beta\oplus\gamma\oplus\delta\in\GL(L).
\]
Moreover, $\sigma_0(U)=\sigma_0(C_1\oplus A_1)=\alpha(C_1)\oplus\beta(A_1)=C_2\oplus A_2=U$, and $\sigma_0(Y_1)=\sigma_0(C_1\oplus B_1)=\alpha(C_1)\oplus\gamma(B_1)=C_2\oplus B_2=Y_2$.

Write $\det(\sigma_0)=\lambda\in \mathbb{F}_q^\times$. Choose a basis $v_1,\ldots,v_{a+b}$ of $L$ adapted to the direct sum decomposition $L=C_2\oplus A_2\oplus B_2\oplus R_2$. Define $t\in\GL(L)$ by
\[
t(v_1)=\lambda^{-1}v_1,\quad t(v_\ell)=v_\ell,\quad 2\leq \ell\leq a+b.
\]
Then $\det(t)=\lambda^{-1}$. Moreover, $t$ stabilizes each of $C_2,A_2,B_2,R_2$, and hence $t(U)=U$, $t(Y_2)=Y_2$. Thus
\[
\sigma:=t\sigma_0\in\SL(L),\quad\sigma(U)=U,\quad\sigma(Y_1)=Y_2.
\]
Extend $\sigma$ to an element $\widehat x\in\SL(V)$ by letting it act as the identity on all $M_s$ with $s\neq i$, as in the proof of Lemma \ref{lem:adjacent-quotient-transitivity}, and put $x:=\pi(\widehat x)\in X$. The lift $\widehat x$ induces $\sigma$ on $L$, so $x$ acts on the subspaces of $L$ by $Y\mapsto \sigma(Y)$. Since $\sigma(U)=U$, we have $xU=U$, and hence $x\in X_U=H$. Moreover, $xY_1=\sigma(Y_1)=Y_2$. Thus $H$ acts transitively on $\Omega_j$.

We next show that $D_j$ is an $H$-double coset. We first note that $D_j\neq\varnothing$. Choose a complement $U'$ of $U$ in $L$, so that $L=U\oplus U'$, $\dim U'=b$. Since $0\leq j\leq\min\{a,b\}$, there exist subspaces $C\leq U$, $B\leq U'$ with $\dim C=a-j$, $\dim B=j$. Then $Y:=C\oplus B$ has dimension $a$, and $U\cap Y=C$. Hence $Y\in\Omega_j$. Since $X$ acts transitively on $\mathrm{Gr}_a(L)$ by Lemma~\ref{lem:adjacent-quotient-transitivity}, there exists $g\in X$ such that $gU=Y$. Thus $\dim(U\cap gU)=a-j$, so $g\in D_j$. Therefore $D_j\neq\varnothing$.

Fix $g\in D_j$. For $h_1,h_2\in H$, since $h_1U=h_2U=U$,
\[
\begin{aligned}
\dim(U\cap h_1gh_2U)=\dim(U\cap h_1gU)=\dim(h_1^{-1}U\cap gU)=\dim(U\cap gU)=a-j.
\end{aligned}
\]
Thus $HgH\subseteq D_j$. Conversely, if $x\in D_j$, then $gU,xU\in\Omega_j$. By the transitivity just proved, there exists $h_1\in H$ such that $h_1gU=xU$. Hence $h_2:=x^{-1}h_1g\in X_U=H$. Therefore $x=h_1gh_2^{-1}\in HgH$. Thus $D_j\subseteq HgH$, and hence $D_j=HgH$.

We next show that these double cosets exhaust $X$. Let $x\in X$, and put $j:=a-\dim(U\cap xU)$. Then $0\leq j\leq a$. Moreover,
\[
a+j=2a-\dim(U\cap xU)=\dim(U+xU)\leq \dim L=a+b,
\]
so $j\leq b$. Hence $0\leq j\leq\min\{a,b\}$, and therefore $x\in D_j$. Thus
\[
X=D_0\mathbin{\dot\cup}D_1\mathbin{\dot\cup}\cdots\mathbin{\dot\cup}D_{\min\{a,b\}}.
\]
Since the sets $D_j$ are defined by distinct values of $\dim(U\cap xU)$, they are pairwise disjoint.

For $j=0$, $x\in D_0\Longleftrightarrow xU=U\Longleftrightarrow x\in H$, so $D_0=H$. Finally, let $g\in D_j$. By Lemma~\ref{lem:grassmann-triangle}, there exists $W\in\operatorname{Gr}_a(L)$ such that $\dim(U\cap W)=\dim(W\cap gU)=a-j$. By Lemma~\ref{lem:adjacent-quotient-transitivity}, there exists $x\in X$ such that $xU=W$. Hence $\dim(U\cap xU)=a-j$ and so $x\in D_j$. On the other hand,
\[
\dim(U\cap x^{-1}gU)=\dim(xU\cap gU)=\dim(W\cap gU)=a-j.
\]
Hence $x^{-1}g\in D_j$, and $g=x(x^{-1}g)\in D_j^2$. Thus $D_j\subseteq D_j^2$.
\end{proof}

The preceding $H$-double-coset decomposition of $X$ allows us to extend a complete mapping of $H$ to one of $X$.
\begin{pro}\label{prop:one-step-flag-lifting}
Let $X$ be the stabilizer in $G$ of a flag $\mathcal F$. Suppose that $U<W$ are two consecutive members of $\mathcal F$, and let $\mathcal F'$ be obtained from $\mathcal F$ by inserting a subspace $Y$ with $U<Y<W$. Let $H$ be the stabilizer of $\mathcal F'$, and put $a=\dim(Y/U)$, $b=\dim(W/Y)$. If $H$ has a complete mapping $\alpha$, then there exists a complete mapping $\theta$ of $X$ such that $\theta|_H=\alpha$.

Moreover, if $\binom{a+b}{a}_q$ is odd and $\alpha$ stabilizes a Sylow $2$-subgroup $P$ of $H$, then $P\in\mathrm{Syl}_2(X)$ and $\theta(P)=P$.
\end{pro}

\begin{proof}
By Lemma~\ref{lem:intersection-double-cosets}, $X=D_0\mathbin{\dot\cup}D_1\mathbin{\dot\cup}\cdots\mathbin{\dot\cup}D_{\min\{a,b\}}$, $D_0=H$. On $D_0=H$, use the complete mapping triple associated with $\alpha$, namely,
\[
A_0(h)=h,\quad B_0(h)=\alpha(h),\quad C_0(h)=h\alpha(h),\quad h\in H.
\]

For each $1\leq j\leq\min\{a,b\}$, Lemma~\ref{lem:intersection-double-cosets} gives $D_j\subseteq D_j^2$. Since $H$ admits a complete mapping, Lemma~\ref{lem:single-double-coset} implies that $D_j$ admits a complete mapping triple. Therefore, by Lemma~\ref{lem:patching}, the above complete mapping triples on $D_0,D_1,\ldots,D_{\min\{a,b\}}$ can be patched together to give a complete mapping triple $(A,B,C)$ on $X$. Let
\[
\theta:=B\circ A^{-1}.
\]
Then $\theta$ is a complete mapping of $X$. On $D_0=H$, the triple $(A,B,C)$ agrees with $(A_0,B_0,C_0)$. Hence, for every $h\in H$, $A(h)=h$, $B(h)=\alpha(h)$, and therefore $\theta(h)=B(A^{-1}(h))=B(h)=\alpha(h)$. Thus $\theta|_H=\alpha$.

Finally, suppose that $[X:H]=\binom{a+b}{a}_q$ is odd and that $P\in\mathrm{Syl}_2(H)$ satisfies $\alpha(P)=P$. Since $[X:H]$ is odd, $X$ and $H$ have the same $2$-part. Hence $P\in\mathrm{Syl}_2(X)$. Moreover, since $P\leq H$ and $\theta|_H=\alpha$, $\theta(P)=\alpha(P)=P$.
\end{proof}

\begin{core}\label{cor:flag-refinement-lifting}
Let $\mathcal F$ and $\mathcal F'$ be flags such that $\mathcal F'$ is obtained from $\mathcal F$ by inserting finitely many subspaces. Let $X:=G_{\mathcal F}$, $H:=G_{\mathcal F'}$. Then every complete mapping $\alpha$ of $H$ extends to a complete mapping $\theta$ of $X$. Moreover, if $[X:H]$ is odd and $\alpha$ stabilizes $P\in\mathrm{Syl}_2(H)$, then $P\in\mathrm{Syl}_2(X)$ and $\theta(P)=P$.
\end{core}

\begin{proof}
Choose a chain of flags $\mathcal F=\mathcal F_0<\mathcal F_1<\cdots<\mathcal F_r=\mathcal F'$, where each $\mathcal F_i$ is obtained from $\mathcal F_{i-1}$ by inserting a single subspace, and let $X_i:=G_{\mathcal F_i}$. Thus
\[
H=X_r\leq X_{r-1}\leq\cdots\leq X_0=X.
\]
Starting with $\alpha$ on $X_r$, repeated application of Proposition \ref{prop:one-step-flag-lifting} yields a complete mapping $\theta$ of $X_0=X$ whose restriction to $H$ is $\alpha$.

Now suppose that $[X:H]$ is odd and that $\alpha$ stabilizes $P\in\mathrm{Syl}_2(H)$. Since $[X:H]=\prod_{i=1}^r [X_{i-1}:X_i]$ is odd, every index $[X_{i-1}:X_i]$ is odd. Hence, applying the second assertion of Proposition \ref{prop:one-step-flag-lifting} successively, $P$ remains a Sylow $2$-subgroup at every step and is stabilized by each extended complete mapping. Therefore $P\in\mathrm{Syl}_2(X)$ and $\theta(P)=P$.
\end{proof}

\subsection{Lifting via two equal-dimensional direct summands of vector space}\label{subsec:equal-block-lifting}

Recall that $\pi:\SL(V)\to G=\SL(V)/Z$ is the natural quotient map, where $Z\leq Z(\SL(V))$. Assume throughout this subsection that $q$ is odd, and let
\[
V=E_1\oplus E_2\oplus F,\quad \dim E_1=\dim E_2=m,\quad \dim F=s,
\]
where $m\geq1$ and $s\geq0$.

In this subsection, we develop a lifting construction associated with this decomposition. It extends a complete mapping from a subgroup of $G$ preserving $F$ and the pair $\{E_1,E_2\}$ setwise to a larger subgroup of $G$ preserving both $E_1\oplus E_2$ and $F$.

We first introduce two subgroups corresponding to the decompositions $V=(E_1\oplus E_2)\oplus F$ and $V=E_1\oplus E_2\oplus F$, respectively. Set
\[
\widehat X:=\bigl\{\diag(A,C)\in\SL(V):A\in\GL(E_1\oplus E_2),\;C\in\GL(F)\bigr\},\quad X:=\pi(\widehat X)
\]
and
\[
\widehat L_0:=\bigl\{\diag(A,B,C)\in\SL(V):A,B\in\GL_m(q),\;C\in\GL_s(q)\bigr\},\quad L_0:=\pi(\widehat L_0).
\]
When $s>0$, Lemma~\ref{lem:flag-Levi} shows that $\widehat X$ and $\widehat L_0$ are Levi subgroups of the stabilizers of the flags $0<E_1\oplus E_2<V$ and $0<E_1<E_1\oplus E_2<V$, respectively. When $s=0$, we have $\widehat X=\SL(V)$, while $\widehat L_0$ is a Levi subgroup of the stabilizer of $0<E_1<V$.

Since $\dim E_1=\dim E_2$, there exists an element of $\widehat X$ that exchanges $E_1$ and $E_2$. We therefore enlarge $L_0$ by adjoining such an element. Let
\[
\widehat K:=\left\{x\in\widehat X:\{xE_1,xE_2\}=\{E_1,E_2\}\right\},\ K:=\pi(\widehat K),\quad \widehat w:=
\begin{pmatrix}
0&-I_m&0\\I_m&0&0\\0&0&I_s
\end{pmatrix},\ w:=\pi(\widehat w).
\]
Then
\[
\widehat K=\widehat L_0\mathbin{\dot\cup}\widehat w\widehat L_0,\quad K=L_0\mathbin{\dot\cup}wL_0,
\]
and hence $[K:L_0]=2$.

The purpose of the construction in this subsection is to extend a prescribed complete mapping $\alpha$ of $K$ to a complete mapping $\theta$ of $X$ such that $\theta|_K=\alpha$. Our strategy is to decompose $X$ into double cosets with respect to a subgroup containing $L_0$. Most of these double cosets will be treated separately using Lemma~\ref{lem:single-double-coset}, while the two double cosets containing $L_0$ and $wL_0$ will be treated together in order to preserve the prescribed complete mapping $\alpha$ on $K$.

For this purpose, we use the stabilizer of $E_1$, which contains $L_0$. Let
\[
Q_0:=X_{E_1}=\{x\in X:xE_1=E_1\}.
\]
To describe the structure of $Q_0$, we introduce the following subgroup. For $Y\in M_m(q)$, set
\[
\widehat u(Y):=
\begin{pmatrix}
I_m&Y&0\\0&I_m&0\\0&0&I_s
\end{pmatrix},\quad
u(Y):=\pi(\widehat u(Y)),
\]
and
\[
\widehat U_0:=\{\widehat u(Y):Y\in M_m(q)\},\quad U_0:=\pi(\widehat U_0)=\{u(Y):Y\in M_m(q)\}.
\]
Since $\widehat u(Y_1)\widehat u(Y_2)=\widehat u(Y_1+Y_2)$ and $\widehat U_0\cap Z=1$, we have $U_0\cong (M_m(q),+)$. Moreover, $L_0$ normalizes $U_0$. Indeed, if $\ell=\pi(\diag(A,B,C))\in L_0$, then $\ell u(Y)\ell^{-1}=u(AYB^{-1})\in U_0$. Every element of $Q_0$ can be written uniquely in the form $u(Y)\ell$, $Y\in M_m(q)$, $\ell\in L_0$. Indeed, if $x\in Q_0$, then a lift of $x$ to $\widehat X$ has the form
\[
\begin{pmatrix}
A&D&0\\0&B&0\\0&0&C
\end{pmatrix}
=\widehat u(DB^{-1})\diag(A,B,C)
\]
with $\diag(A,B,C)\in\widehat L_0$. Since $U_0\cap L_0=1$ and $L_0$ normalizes $U_0$, it follows that
\[
Q_0=U_0\rtimes L_0.
\]

We now define the subsets that will form the $Q_0$-double-coset decomposition of $X$. For $0\leq j\leq m$, define
\[
\mathcal D_j:=\{x\in X:\dim(E_1\cap xE_1)=m-j\}.
\]

\begin{lem}\label{lem:equal-block-double-cosets}
Suppose that $L_0$ admits a complete mapping. Then $Q_0$ admits a complete mapping, and $X=\mathcal D_0\mathbin{\dot\cup}\mathcal D_1\mathbin{\dot\cup}\cdots\mathbin{\dot\cup}\mathcal D_m$, where each $\mathcal D_j$ is a $Q_0$-double coset, $\mathcal D_0=Q_0$, $\mathcal D_m=Q_0wQ_0$. Moreover, for every $1\leq j\leq m-1$, the set $\mathcal D_j$ admits a complete mapping triple.
\end{lem}

\begin{proof}
We have $Q_0=U_0\rtimes L_0$, where $U_0\cong(M_m(q),+)$. Since $q$ is odd, $|U_0|=q^{m^2}$ is odd, and hence $U_0$ admits a complete mapping. Moreover, by hypothesis, $L_0$ admits a complete mapping. Since $Q_0=U_0L_0$ is an exact factorization, Lemma \ref{lem:exact-factorization} implies that $Q_0$ admits a complete mapping.

As in the proof of Lemma~\ref{lem:adjacent-quotient-transitivity}, $\SL(E_1\oplus E_2)$ acts transitively on the $m$-dimensional subspaces of $E_1\oplus E_2$. Since this action is contained in that of $X$, the group $X$ is also transitive on these subspaces. Moreover, $Q_0=X_{E_1}$. Hence, by the argument of Lemma~\ref{lem:intersection-double-cosets},
\[
X=\mathcal D_0\mathbin{\dot\cup}\mathcal D_1\mathbin{\dot\cup}\cdots\mathbin{\dot\cup}\mathcal D_m,
\]
where each $\mathcal D_j$ is a $Q_0$-double coset and $\mathcal D_j\subseteq\mathcal D_j^2$, $0\leq j\leq m$.

By definition, $\mathcal D_0=\{x\in X:xE_1=E_1\}=Q_0$. Since $wE_1=E_2$ and $E_1\cap E_2=0$, we have $w\in\mathcal D_m$. As $\mathcal D_m$ is a $Q_0$-double coset, it follows that $\mathcal D_m=Q_0wQ_0$.

Finally, for $1\leq j\leq m-1$, since $\mathcal D_j\subseteq\mathcal D_j^2$ and $Q_0$ admits a complete mapping, Lemma \ref{lem:single-double-coset} implies that $\mathcal D_j$ admits a complete mapping triple.
\end{proof}

By Lemma~\ref{lem:equal-block-double-cosets}, each of $\mathcal D_1,\ldots,\mathcal D_{m-1}$ admits a complete mapping triple. It remains to treat $\mathcal D_0$ and $\mathcal D_m$. Since $K=L_0\mathbin{\dot\cup}wL_0$, $L_0\subseteq\mathcal D_0$, $wL_0\subseteq\mathcal D_m$, we have $K\subseteq \mathcal D_0\mathbin{\dot\cup}\mathcal D_m$. Thus, in order to preserve the prescribed complete mapping $\alpha$ on $K$, $\mathcal D_0$ and $\mathcal D_m$ will be treated together. Put
\[
S:=\mathcal D_0\mathbin{\dot\cup}\mathcal D_m=Q_0\mathbin{\dot\cup}Q_0wQ_0.
\]
We now describe the $U_0$-cosets contained in $S$, which will be used in the application of Lemma~\ref{lem:local-wilcox}.

Since $Q_0=U_0\rtimes L_0$, the $U_0$-cosets contained in $Q_0$ are $\ell U_0$, $\ell\in L_0$. For the double coset $Q_0wQ_0$, since $w^{-1}L_0w=L_0$ and $L_0$ normalizes $U_0$, $Q_0wQ_0=U_0L_0wU_0L_0=U_0wL_0U_0$. Hence its $U_0$-cosets are of the form
\[
B_{Y,\ell}:=u(Y)w\ell U_0,\quad Y\in M_m(q),\quad \ell\in L_0.
\]
This parametrization is unique. Indeed, suppose that $B_{Y,\ell}=B_{Y',\ell'}$. Then $u(Y)w\ell=u(Y')w\ell'u(T)$ for some $u(T)\in U_0$, and hence
\[
u(Y-Y')=w\ell'u(T)\ell^{-1}w^{-1}\in Q_0\cap wQ_0w^{-1}.
\]
Since $wE_1=E_2$, $Q_0\cap wQ_0w^{-1}=X_{E_1}\cap X_{E_2}=L_0$. Thus $u(Y-Y')\in U_0\cap L_0=1$, and so $Y=Y'$. It then follows from $w\ell U_0=w\ell'U_0$ that $\ell^{-1}\ell'\in U_0\cap L_0=1$, whence $\ell=\ell'$.

Finally, since $B_{0,\ell}=w\ell U_0$, and $K=L_0\mathbin{\dot\cup}wL_0$, the $U_0$-cosets represented by elements of $K$ are precisely
\[
\{\ell U_0:\ell\in L_0\}\mathbin{\dot\cup}\{B_{0,\ell}:\ell\in L_0\}=\{hU_0:h\in K\}.
\]
Thus the $U_0$-cosets contained in $S$ split into the $U_0$-cosets represented by elements of $K$, namely $\{hU_0:h\in K\}$, and the cosets $B_{Y,\ell}$ with $Y\neq0$ and $\ell\in L_0$.

\begin{lem}\label{lem:equal-block-coset-layer}
Suppose that $K$ has a prescribed complete mapping $\alpha$, and that $L_0$ admits a complete mapping. Then there exist an index set $I$ and three bijections
\[
\widetilde x,\widetilde y,\widetilde z:I\longrightarrow S/U_0
\]
such that $\widetilde z(i)\subseteq\widetilde x(i)\widetilde y(i)$, $i\in I$. Moreover, the index set $I$ may be chosen so that $K\subseteq I$ and, for every $h\in K$, $\widetilde x(h)=hU_0$, $\widetilde y(h)=\alpha(h)U_0$, $\widetilde z(h)=h\alpha(h)U_0$.
\end{lem}

\begin{proof}
Put $\mathcal M:=(M_m(q),+)$, $I:=K\mathbin{\dot\cup}\bigl((\mathcal M\setminus\{0\})\times L_0\bigr)$. We define $\widetilde x,\widetilde y,\widetilde z$ separately on the two parts of $I$.

We first deal with the $K$-part. For $h\in K$, define
\[
\widetilde x(h)=hU_0,\quad\widetilde y(h)=\alpha(h)U_0,\quad\widetilde z(h)=h\alpha(h)U_0.
\]
By the preceding description of the $U_0$-cosets represented by elements of $K$, the map $h\mapsto hU_0$ is a bijection from $K$ onto $\{hU_0:h\in K\}$. Since both $\alpha$ and $h\mapsto h\alpha(h)$ are permutations of $K$, each of the three families above runs bijectively through these $U_0$-cosets. Moreover, $\widetilde z(h)=h\alpha(h)U_0\subseteq (hU_0)(\alpha(h)U_0)=\widetilde x(h)\widetilde y(h)$.

It remains to treat the cosets $B_{Y,\ell}$, $Y\neq0$. Consider the Cayley graph
\[
\Gamma:=\operatorname{Cay}(\mathcal M,\GL_m(q)),\quad Y\sim Y'\Longleftrightarrow Y-Y'\in\GL_m(q).
\]
Then $\Gamma$ is vertex-transitive. Moreover, it is connected if and only if $\GL_m(q)$ generates the group $\mathcal M$. For every $c\in\F_q^\times$, we have $cE_{ij}=cI_m-(cI_m-cE_{ij})$, $i\neq j$, and $cE_{ii}=2cI_m-(2cI_m-cE_{ii})$. Since $q$ is odd, all matrices on the right-hand sides are invertible. Hence $cE_{ij}\in \langle \GL_m(q)\rangle_+$, $c\in\F_q^\times$, $1\leq i,j\leq m$. Since every matrix in $M_m(q)$ is a sum of matrices of the form $cE_{ij}$, it follows that $\langle \GL_m(q)\rangle_+=\mathcal M$. Thus $\Gamma$ is connected. Since $\Gamma$ is connected and vertex-transitive and has odd order $q^{m^2}$, the graph $\Gamma-\{0\}$ has a perfect matching by \cite[Theorem~1.1(a)]{ZhangSun}. Hence there exists a fixed-point-free involution $\tau:\mathcal M\setminus\{0\}\longrightarrow \mathcal M\setminus\{0\}$ such that $Y$ and $\tau(Y)$ are adjacent for every $Y\neq0$. Thus
\[
\Delta_Y:=Y-\tau(Y)\in\GL_m(q),\quad Y\neq0.
\]

Choose a complete mapping $\phi:L_0\longrightarrow L_0$, and put $\widehat\phi(t):=t\phi(t)$. Then both $\phi$ and $\widehat\phi$ are permutations of $L_0$. Fix $Y\neq0$. For $\ell\in L_0$, choose a lift $\widehat\ell=\diag(A,B,C)\in\widehat L_0$ and define
\[
s_Y(\ell):=\pi\left(\diag(\Delta_Y^{-1}B,\Delta_YA,C)\right).
\]
Since $\det(\Delta_Y^{-1}B)\det(\Delta_YA)\det C=\det A\det B\det C=1$, $\diag(\Delta_Y^{-1}B,\Delta_YA,C)\in \widehat L_0$. Moreover, the map $s_Y$ is well defined, since replacing $\widehat\ell$ by another lift multiplies $A,B,C$ by the same central scalar and therefore leaves the same image under $\pi$. Finally, $s_Y^2(\ell)=\pi(\diag(A,B,C))=\ell$. Thus $s_Y^2=\mathrm{id}_{L_0}$, and hence $s_Y$ is a permutation of $L_0$.

For $Y\neq0$, $\ell\in L_0$, define
\[
\begin{aligned}
\widetilde x(Y,\ell)&:=B_{Y,\ell},\quad
\widetilde y(Y,\ell):=B_{Y,\phi(s_Y(\ell))},\quad
\widetilde z(Y,\ell):=B_{\tau(Y),\widehat\phi(s_Y(\ell))}.
\end{aligned}
\]
Since $\tau,s_Y,\phi$, and $\widehat\phi$ are permutations, each of these three families runs bijectively through the cosets $B_{Y,\ell}$, $Y\neq0$, $\ell\in L_0$.

It remains to verify that $\widetilde z(Y,\ell)\subseteq\widetilde x(Y,\ell)\widetilde y(Y,\ell)$, $Y\neq0$, $\ell\in L_0$. We first establish a general product inclusion. Let $\widehat\ell=\diag(A,B,C)$, $\widehat\ell'=\diag(A',B',C')$ be lifts of $\ell,\ell'\in L_0$. For $R\in\GL_m(q)$, choose
\[
u(Y)w\ell u(R-Y')\in B_{Y,\ell},\quad u(Y')w\ell'\in B_{Y',\ell'}.
\]
Since $u(R-Y')u(Y')=u(R)$, their product is $u(Y)w\ell u(R)w\ell'$. A direct calculation then gives
\begin{equation}\label{eq:BYl}
B_{Y,\ell}B_{Y',\ell'}
\supseteq
B_{Y-BR^{-1}A^{-1},\,\pi\left(\diag(ARA',BR^{-1}B',CC')\right)}.
\end{equation}
We now choose $R$ and $\ell'$ so that the right-hand side of \eqref{eq:BYl} is precisely $\widetilde z(Y,\ell)$. Set $R=A^{-1}\Delta_Y^{-1}B$. Then $R^{-1}=B^{-1}\Delta_YA$, and hence $Y-BR^{-1}A^{-1}=Y-\Delta_Y=\tau(Y)$. Next set $\ell'=\phi(s_Y(\ell))$. By the choice of $R$, $AR=\Delta_Y^{-1}B$, $BR^{-1}=\Delta_YA$. Therefore, by the definition of $s_Y$,
\[
\begin{aligned}
\pi\left(\diag(ARA',BR^{-1}B',CC')\right)=s_Y(\ell)\phi(s_Y(\ell))=\widehat\phi(s_Y(\ell)).
\end{aligned}
\]
Hence \eqref{eq:BYl} yields $B_{\tau(Y),\,\widehat\phi(s_Y(\ell))}\subseteq B_{Y,\ell}\, B_{Y,\phi(s_Y(\ell))}$, that is, $\widetilde z(Y,\ell)\subseteq\widetilde x(Y,\ell)\widetilde y(Y,\ell)$.

Together, the two constructions give the required bijections $\widetilde x,\widetilde y,\widetilde z:I\longrightarrow S/U_0$, satisfying the required product condition.
\end{proof}

To apply Lemma~\ref{lem:local-wilcox}, we also need representatives of the $U_0$-cosets that form simultaneously a left and right transversal in $S=\mathcal D_0\mathbin{\dot\cup}\mathcal D_m=Q_0\mathbin{\dot\cup}Q_0wQ_0$. The next lemma provides one.

\begin{lem}\label{lem:equal-block-transversal}
Let $\mathcal R:=L_0\mathbin{\dot\cup}\bigl\{u(Y)w\ell u(Y):Y\in M_m(q),\ell\in L_0\bigr\}$. Then $\mathcal R$ is simultaneously a transversal for the left and right $U_0$-cosets contained in $S$. Moreover, $K\subseteq \mathcal R$.
\end{lem}

\begin{proof}
We first show that $\mathcal R$ is a transversal for the left $U_0$-cosets contained in $S$. Since $Q_0=U_0\rtimes L_0$, the set $L_0$ is a transversal for the left $U_0$-cosets in $Q_0$. Moreover,
\[
u(Y)w\ell u(Y)U_0=u(Y)w\ell U_0=B_{Y,\ell}.
\]
Since the $B_{Y,\ell}$, $Y\in M_m(q)$, $\ell\in L_0$, are precisely the left $U_0$-cosets contained in $Q_0wQ_0$, and their parametrization is unique, the elements $u(Y)w\ell u(Y)$, $Y\in M_m(q)$, $\ell\in L_0$, represent these cosets exactly once. Hence $\mathcal R$ is a left $U_0$-transversal in $S$.

We next show that $\mathcal R$ is also a transversal for the right $U_0$-cosets contained in $S$. For $\ell\in L_0$, let $\ell^\dagger:=w^{-1}\ell^{-1}w^{-1}$. Since $w$ normalizes $L_0$ and $w^2\in L_0$, $\ell^\dagger=(w^{-1}\ell^{-1}w)w^{-2}\in L_0$. Therefore
\[
\begin{aligned}
\bigl(u(Y)w\ell u(Y)\bigr)^{-1}=u(-Y)\ell^{-1}w^{-1}u(-Y)=u(-Y)w\ell^\dagger u(-Y).
\end{aligned}
\]
Since $L_0^{-1}=L_0$, the preceding formula shows that $\mathcal R^{-1}=\mathcal R$.

Moreover, $S^{-1}=S$. Indeed, $Q_0^{-1}=Q_0$, and since $w^2\in L_0\leq Q_0$, we have $w^{-1}=w\,w^{-2}\in wQ_0$. Thus $(Q_0wQ_0)^{-1}=Q_0w^{-1}Q_0=Q_0wQ_0$, and consequently, $S^{-1}=S$. Therefore, for every left $U_0$-coset $rU_0\subseteq S$, $(rU_0)^{-1}=U_0r^{-1}\subseteq S^{-1}=S$. Thus inversion maps the left $U_0$-cosets contained in $S$ bijectively onto the right $U_0$-cosets contained in $S$. Hence $\mathcal R^{-1}=\mathcal R$ is a transversal for the right $U_0$-cosets contained in $S$.

Finally, $L_0\subseteq\mathcal R$, and for every $\ell\in L_0$, taking $Y=0$ gives $w\ell=u(0)w\ell u(0)\in\mathcal R$. Hence $L_0\mathbin{\dot\cup}wL_0=K\subseteq\mathcal R$.
\end{proof}

We are now ready to establish the desired lifting from $K$ to $X$.

\begin{pro}\label{pro:relative-equal-block-lifting}
Suppose that $K$ has a prescribed complete mapping $\alpha$, and that $L_0$ admits a complete mapping. Then there exists a complete mapping $\theta$ of $X$ such that $\theta|_K=\alpha$.
\end{pro}

\begin{proof}
By Lemma~\ref{lem:equal-block-double-cosets}, each $\mathcal D_j$, $1\leq j\leq m-1$, admits a complete mapping triple. Since $K=L_0\mathbin{\dot\cup}wL_0\subseteq\mathcal D_0\mathbin{\dot\cup}\mathcal D_m$, the prescribed complete mapping $\alpha$ only needs to be retained on the $\mathcal D_0$ and $\mathcal D_m$. Thus it remains to construct a complete mapping triple on $S:=\mathcal D_0\mathbin{\dot\cup}\mathcal D_m=Q_0\mathbin{\dot\cup}Q_0wQ_0$ that preserves $\alpha$ on $K$.

By Lemma~\ref{lem:equal-block-coset-layer}, there exist bijections $\widetilde x,\widetilde y,\widetilde z:I\longrightarrow S/U_0$ such that $\widetilde z(i)\subseteq\widetilde x(i)\widetilde y(i)$, $i\in I$, and, for $h\in K$, $\widetilde x(h)=hU_0$, $\widetilde y(h)=\alpha(h)U_0$, $\widetilde z(h)=h\alpha(h)U_0$.

By Lemma~\ref{lem:equal-block-transversal}, there exists a set $\mathcal R$ which is simultaneously a left and right $U_0$-transversal in $S$, and $K\subseteq \mathcal R$. For each $i\in I$, let $y_i$ be the unique element of $\mathcal R\cap\widetilde y(i)$. If $h\in K$, then $\widetilde y(h)=\alpha(h)U_0$ and $\alpha(h)\in K\subseteq\mathcal R$. Hence $y_h=\alpha(h)$.

Since the identity map on $U_0$ is a complete mapping, choose an associated complete mapping triple $a,b,c:J\to U_0$ with an index $j_0\in J$ satisfying $a(j_0)=b(j_0)=c(j_0)=1$. We may therefore apply Lemma~\ref{lem:local-wilcox}. For each $h\in K$, prescribe $x_h:=h$, $z_h:=h\alpha(h)$. As shown above, $y_h=\alpha(h)$. Then $x_h\in\widetilde x(h)$, $y_h\in\widetilde y(h)$, $z_h\in\widetilde z(h)$, and $z_h=x_hy_h$. Hence Lemma~\ref{lem:local-wilcox} yields a complete mapping triple $(A_S,B_S,C_S)$ on $S$ such that $A_S(h,j_0)=h$, $B_S(h,j_0)=\alpha(h)$, $h\in K$.

Finally, apply Lemma~\ref{lem:patching} to $(A_S,B_S,C_S)$ and the triples on the remaining double cosets. This gives a complete mapping triple $(A,B,C)$ on $X$ extending $(A_S,B_S,C_S)$.

Let $\theta:=B\circ A^{-1}$. Then $\theta$ is a complete mapping of $X$. Moreover, for every $h\in K$, $\theta(h)=B(A^{-1}(h))=B(h,j_0)=\alpha(h)$. Therefore $\theta|_K=\alpha$.
\end{proof}

\section{Over finite fields of odd characteristic}\label{sec:odd}

Throughout this section, assume that $q$ is odd and let $V=\F_q^n$. Retaining the notation of the preceding section, set $G=\SL_n(q)/Z$, $Z\leq Z(\SL_n(q))$, and let $\pi:\SL_n(q)\to G$ be the natural quotient map.

Our aim is to prove that $G$ admits a complete mapping stabilizing a Sylow $2$-subgroup. More precisely, for $n\geq2$, let $\mathsf{CM}(n)$ denote the following statement. We shall prove $\mathsf{CM}(n)$ for all $n\geq2$.

\begin{statement}
$\mathsf{CM}(n)$: for every $Z\leq Z(\SL_n(q))$, the group $\SL_n(q)/Z$ admits a complete mapping stabilizing a Sylow $2$-subgroup.
\end{statement}

The proof uses the two lifting constructions developed in Section~\ref{sec:linear-lifting}. We first establish $\mathsf{CM}(2)$. We then apply Proposition~\ref{pro:relative-equal-block-lifting} to prove $\mathsf{CM}(n)$ when $n$ is a power of $2$. Finally, for arbitrary $n$, we use its binary expansion to construct a suitable flag and apply Proposition~\ref{prop:one-step-flag-lifting}. The case $n=3$ with $q\equiv3\pmod4$ requires a separate argument, since the equal-dimensional decomposition used in the general construction does not yield a subgroup of odd index in this case. The parity results of Lemmas~\ref{lem:qbinom-parity} and \ref{lem:equal-qbinom-valuation} ensure that the Sylow $2$-subgroups are preserved in these lifting steps.

\subsection{The case $n=2$}

We first prove $\mathsf{CM}(2)$, starting with a lemma valid for both $q\equiv1,3\pmod4$.

\begin{lem}\label{lem:rank-one-normalizer}
Assume that $n=2$ and $q$ is odd. Let $\widehat T^{+}$ and $\widehat T^{-}$ be as defined in Definitions~\ref{def:Borel} and~\ref{def:nonsplit-torus}, respectively. Set
\[
\widehat T=
\begin{cases}
\widehat T^{+}, & q\equiv1\pmod4\\
\widehat T^{-}, & q\equiv3\pmod4
\end{cases},\quad T:=\pi(\widehat T),\quad \widehat N:=N_{\SL_2(q)}(\widehat T),\quad N:=\pi(\widehat N).
\]
Then $N$ contains some $P\in\mathrm{Syl}_2(G)$ and admits a complete mapping stabilizing $P$.

Moreover, if $q\equiv1\pmod4$, then $N=T^{+}\mathbin{\dot\cup}wT^{+}$, where $w=\pi\begin{pmatrix}0&-1\\1&0\end{pmatrix}$.
\end{lem}

\begin{proof}
Since $Z(\SL_2(q))=\{\pm I_2\}$, we have $Z\leq\widehat T$. By \cite[(1.1.6), (1.1.13), and Corollary~1.3.2]{Bonnafe2011}, $\widehat N/\widehat T\cong C_2$. Since $Z\leq\widehat T$, $N/T\cong\widehat N/\widehat T\cong C_2$, and hence $[N:T]=2$. Moreover, $|\widehat T|=q-1$ or $q+1$ according as $q\equiv1$ or $3\pmod4$. Hence
\[
[G:N]=\frac{q(q+1)}{2}\quad\text{or}\quad\frac{q(q-1)}{2},
\]
respectively. In either case, $[G:N]$ is odd.

We next show that $N$ contains a noncyclic $2$-subgroup. Since $4\mid|\widehat T|$ and $\widehat T$ is cyclic, choose $\widehat t\in\widehat T$ of order $4$. Since $\widehat t$ has order $4$, its square is the unique involution of $\widehat T$, namely $-I_2$. Suppose first that $q\equiv1\pmod4$. Set
\[
\widehat w:=
\begin{pmatrix}
0&-1\\
1&0
\end{pmatrix},\quad
w:=\pi(\widehat w).
\]
Then $\widehat w\in\widehat N\setminus\widehat T^{+}$, and $\widehat w^{\,2}=-I_2$, $\widehat w\widehat t\widehat w^{-1}=\widehat t^{-1}$. Thus $\langle\widehat t,\widehat w\rangle\cong Q_8$. Now suppose that $q\equiv3\pmod4$. By \cite[(1.1.12)]{Bonnafe2011}, there exists $\widehat s\in\widehat N\setminus\widehat T^{-}$ such that $\widehat s^{\,2}=-I_2$, $\widehat s\widehat t\widehat s^{-1}=\widehat t^{-1}$. Hence $\langle\widehat t,\widehat s\rangle\cong Q_8$. Thus, in either case, $\widehat N$ contains a subgroup $\widehat Q\cong Q_8$ whose central involution is $-I_2$. If $Z=1$, then $\pi(\widehat Q)\cong Q_8$. If $Z=\{\pm I_2\}$, then $\pi(\widehat Q)\cong Q_8/\{\pm I_2\}\cong C_2\times C_2$. Therefore $N$ contains a noncyclic $2$-subgroup.

Choose $P\in\mathrm{Syl}_2(N)$ containing $\pi(\widehat Q)$. Then $P$ is noncyclic. Since $[G:N]$ is odd, we also have $P\in\mathrm{Syl}_2(G)$.

We now construct a complete mapping of $N$ stabilizing $P$. Since $T$ is cyclic, let $R$ be its unique largest odd-order subgroup, and write $|T|=2^a r$, $r$ odd. Since $[N:T]=2$ and $P\in\mathrm{Syl}_2(N)$, we have
\[
|R|=r,\quad|P|=2^{a+1}.
\]
Hence $R\cap P=1$ and $|RP|=|R||P|=|N|$. Thus $N=RP$ is an exact factorization. The group $R$ has odd order, so it admits a complete mapping. The group $P$ is a noncyclic $2$-group, so it admits a complete mapping by Theorem \ref{thm:HP}. Hence Lemma~\ref{lem:exact-factorization} extends a prescribed complete mapping of $P$ to a complete mapping of $N$ stabilizing $P$.
	
Finally, suppose that $q\equiv1\pmod4$. With $\widehat w$ and $w$ as above, $\widehat w\in\widehat N\setminus\widehat T^{+}$. Since $Z\leq\widehat T^{+}$, it follows that $w\in N\setminus T^{+}$. As $[N:T^{+}]=2$, $N=T^{+}\mathbin{\dot\cup}wT^{+}$.
\end{proof}

For later use, let, as in Definition~\ref{def:Borel},
\[
\widehat U=
\left\{
\begin{pmatrix}
1&x\\0&1
\end{pmatrix}:x\in\F_q
\right\},\quad
\widehat T^{+}=
\left\{
\begin{pmatrix}
a&0\\0&a^{-1}
\end{pmatrix}:a\in\F_q^\times
\right\},
\]
For $x\in\F_q$ and $a\in\F_q^\times$, write
\[
u(x):=\pi
\begin{pmatrix}
1&x\\0&1
\end{pmatrix},\quad
h(a):=\pi
\begin{pmatrix}
a&0\\0&a^{-1}
\end{pmatrix},
\]
and set
\[
U:=\pi(\widehat U)=\{u(x):x\in\F_q\},\quad T^{+}:=\pi(\widehat T^{+})=\{h(a):a\in\F_q^\times\}.
\]
Since $\widehat U\cap Z=1$ and $Z\leq\widehat T^{+}$, we have
\[
|U|=q,\quad|T^{+}|=\frac{q-1}{|Z|}.
\]

\begin{pro}\label{prop:n2-q3mod4}
Assume that $q\equiv3\pmod4$. Then $\mathsf{CM}(2)$ holds.
\end{pro}

\begin{proof}
By Lemma \ref{lem:rank-one-normalizer}, $N$ contains some $P\in\mathrm{Syl}_2(G)$ and admits a complete mapping $\alpha$ satisfying $\alpha(P)=P$. Moreover, $[G:N]=q(q-1)/2$.

Since $|Z|\in\{1,2\}$, the cyclic group $T^{+}$ contains a subgroup $T_0$ of order $(q-1)/2$. By Definition~\ref{def:Borel}, $\widehat T^{+}$ acts on $\widehat U$ by conjugation. Hence $T^{+}=\pi(\widehat T^{+})$ normalizes $U=\pi(\widehat U)$, and so does $T_0$. Thus $H:=UT_0$ is a subgroup of $G$. Since $\gcd(|U|,|T_0|)=\gcd\left(q,\frac{q-1}{2}\right)=1$, $U\cap T_0=1$. Hence $|H|=|U||T_0|=q(q-1)/2$. As $q\equiv3\pmod4$, $|H|$ is odd, so $H$ admits a complete mapping.

On the other hand, $|N|=|G|/[G:N]=2(q+1)/|Z|$. Since $q\equiv3\pmod4$, $\gcd(|H|,|N|)=1$, $|H||N|=|G|$. Therefore $G=HN$, $H\cap N=1$, so $G=HN$ is an exact factorization. By Lemma~\ref{lem:exact-factorization}, the prescribed complete mapping $\alpha$ of $N$ extends to a complete mapping $\theta$ of $G$. Since $P\leq N$, $\theta(P)=\alpha(P)=P$. Hence $\mathsf{CM}(2)$ holds.
\end{proof}

We now turn to the case $q\equiv1\pmod4$. Retain the notation $w$ from Lemma~\ref{lem:rank-one-normalizer}, and set
\[
Z_0:=\{c\in\F_q^\times:cI_2\in Z\},\quad\Lambda:=\F_q^\times/Z_0.
\]
Write $\bar a:=aZ_0\in\Lambda$. Then $|Z_0|=|Z|$, and the map $\F_q^\times\longrightarrow T^+$, $a\mapsto h(a)$, has kernel $Z_0$. Hence $T^+\cong\F_q^\times/Z_0=\Lambda$, and $h(a)$ depends only on the coset $\bar a=aZ_0$.

By Lemma~\ref{lem:rank-one-normalizer}, $N=T^+\mathbin{\dot\cup}wT^+$, and $N$ admits a complete mapping stabilizing a Sylow $2$-subgroup of $G$. Our strategy is to extend this complete mapping from $N$ to $G$. We shall apply Lemma~\ref{lem:local-wilcox} with $H=U$ and $S=G$. To this end, we first describe the relevant $U$-cosets and construct the required bijections on $G/U$, incorporating the prescribed complete mapping of $N$, and then construct a simultaneous left and right $U$-transversal of $G$.

\begin{lem}\label{lem:rank-one-cosets}
Let $N$ be as in Lemma~\ref{lem:rank-one-normalizer}. For $x\in\F_q$ and $\bar a\in\Lambda$, set $A_{\bar a}:=h(a)U$, $B_{x,\bar a}:=u(x)wh(a)U$. Then
\[
G/U=\{A_{\bar a}:\bar a\in\Lambda\}\mathbin{\dot\cup}\{B_{x,\bar a}:x\in\F_q,\ \bar a\in\Lambda\}.
\]
Among these, the $U$-cosets arising from $N$ are precisely $\{gU:g\in N\}=\{A_{\bar a},B_{0,\bar a}:\bar a\in\Lambda\}$. Finally, for $x\in\F_q$ and $a,b,\lambda\in\F_q^\times$,
\[
B_{x-\frac{1}{a^2\lambda},\,\overline{ab\lambda}}\subseteq B_{x,\bar a}B_{\lambda,\bar b}.
\]
\end{lem}

\begin{proof}
Since $Z\leq\widehat T^{+}$, applying $\pi$ to the Bruhat decomposition in Lemma~\ref{lem:SL2-Bruhat} gives
\[
G=T^{+}U\mathbin{\dot\cup}UwT^{+}U.
\]
Hence every left $U$-coset of $G$ has the form $A_{\bar a}$ or $B_{x,\bar a}$.

We show that these representations are unique. If $A_{\bar a}=A_{\bar b}$, then $h(a)^{-1}h(b)\in T^{+}\cap U$. Since $|U|=q$ and $|T^{+}|=(q-1)/|Z|$ are coprime, $U\cap T^{+}=1$. Hence $h(a)=h(b)$, and so $\bar a=\bar b$. If $B_{x,\bar a}=B_{y,\bar b}$, then for some $z\in\F_q$, $u(x)wh(a)=u(y)wh(b)u(z)$ in $G$. Hence, for some $c\in Z_0$,
\[
\begin{pmatrix}1&x\\0&1\end{pmatrix}
\widehat w
\begin{pmatrix}a&0\\0&a^{-1}\end{pmatrix}
=cI_2
\begin{pmatrix}1&y\\0&1\end{pmatrix}
\widehat w
\begin{pmatrix}b&0\\0&b^{-1}\end{pmatrix}
\begin{pmatrix}1&z\\0&1\end{pmatrix}.
\]
Comparing entries gives $z=0$, $x=y$, and $a=cb$. Thus $\bar a=\bar b$. The two families are disjoint by Lemma \ref{lem:SL2-Bruhat}.

We next determine the $U$-cosets arising from $N$. By Lemma~\ref{lem:rank-one-normalizer}, $N=T^{+}\mathbin{\dot\cup}wT^{+}$. Hence an element $h(a)\in T^{+}$ determines the $U$-coset $h(a)U=A_{\bar a}$, while every element of $wT^+$ has the form $wh(a)$ and determines the $U$-coset $wh(a)U=B_{0,\bar a}$. Hence $\{gU:g\in N\}=\{A_{\bar a},B_{0,\bar a}:\bar a\in\Lambda\}$.

Finally, a direct calculation gives
\[
u(x)wh(a)u(\lambda)wh(b)=u\left(x-\frac{1}{a^2\lambda}\right)wh(ab\lambda)u\left(-\frac{1}{b^2\lambda}\right).
\]
Since $u\left(-\frac{1}{b^2\lambda}\right)\in U$, the right-hand side lies in $B_{\,x-\frac{1}{a^2\lambda},\,\overline{ab\lambda}}$. Moreover, $B_{x,\bar a}B_{\lambda,\bar b}$ is invariant under right multiplication by $U$. Hence
\[
B_{\,x-\frac{1}{a^2\lambda},\,\overline{ab\lambda}}
\subseteq
B_{x,\bar a}B_{\lambda,\bar b}.\qedhere
\]
\end{proof}

The $U$-cosets arising from $N$ will be handled by the prescribed complete mapping of $N$. We next consider the $U$-cosets outside $N$, namely $B_{x,\bar a}$ with $x\in\F_q^\times$. For this purpose, we first need the following auxiliary choice.

\begin{lem}\label{lem:rank-one-choice}
For each $\bar a=aZ_0\in\Lambda$, one can choose $t_{\bar a}\in\F_q^\times$ such that $1+a^2t_{\bar a}\neq0$, and the map $\rho:\Lambda\longrightarrow\Lambda$, $\bar a\mapsto\overline{t_{\bar a}}$, is a permutation.
\end{lem}

\begin{proof}
If $|Z_0|=2$, then $aZ_0$ has two elements, so we may choose $t_{\bar a}\in aZ_0$ with $t_{\bar a}\neq-a^{-2}$. Then $1+a^2t_{\bar a}\neq0$ and $\rho(\bar a)=\overline{t_{\bar a}}=\bar a$. Hence $\rho=\mathrm{id}_{\Lambda}$ is a permutation.

Suppose that $Z_0=1$, so $\Lambda=\F_q^\times$. Let $\Gamma$ be the bipartite graph with two copies of $\F_q^\times$, where $a\sim t$ if $t\neq-a^{-2}$. For $A\subseteq\F_q^\times$, if $A=\varnothing$, then $|N_\Gamma(A)|=0=|A|$, so the condition in Hall's theorem \cite[Theorem~2.1.2]{Diestel2017} is satisfied. Hence we may assume that $A\neq\varnothing$. If $-a^{-2}\neq-b^{-2}$ for some $a,b\in A$, then
$$
N_\Gamma(A)=\bigcup_{c\in A}\bigl(\F_q^\times\setminus\{-c^{-2}\}\bigr)\supseteq\bigl(\F_q^\times\setminus\{-a^{-2}\}\bigr)\cup\bigl(\F_q^\times\setminus\{-b^{-2}\}\bigr)=\F_q^\times.
$$
If no such pair exists, then $-a^{-2}=-b^{-2}$ for all $a,b\in A$. Thus $A\subseteq\{a_0,-a_0\}$ for some $a_0\in A$, so $|A|\leq2$, while $N_\Gamma(A)=\F_q^\times\setminus\{-a_0^{-2}\}$ has size $q-2\geq2$. Hence $|N(A)|\geq|A|$ in all cases. By Hall's theorem \cite[Theorem~2.1.2]{Diestel2017}, $\Gamma$ has a perfect matching. If $t_a$ is matched to $a$, then $t_a\neq-a^{-2}$, so $1+a^2t_a\neq0$. Since the matching is perfect, the map $\rho:\Lambda\longrightarrow\Lambda$, $\rho(a)=t_a$, is a permutation.
\end{proof}

We now construct the three bijections on $G/U$ required for Lemma~\ref{lem:local-wilcox}.

\begin{lem}\label{lem:rank-one-coset-maps}
Let $\alpha$ be a prescribed complete mapping of $N$. Then there exist an index set $I$ and bijections $\widetilde x,\widetilde y,\widetilde z:I\longrightarrow G/U$ such that $\widetilde z(i)\subseteq\widetilde x(i)\widetilde y(i)$, $i\in I$. Moreover, we may take $N\subseteq I$ so that, for every $h\in N$, $\widetilde x(h)=hU$, $\widetilde y(h)=\alpha(h)U$, $\widetilde z(h)=h\alpha(h)U$.
\end{lem}

\begin{proof}
Put $I_1:=\F_q^\times\times\Lambda$, $I:=N\mathbin{\dot\cup}I_1$. We define the three bijections separately on $N$ and $I_1$.

For $h\in N$, set
\[
\widetilde x(h):=hU,\quad\widetilde y(h):=\alpha(h)U,\quad\widetilde z(h):=h\alpha(h)U.
\]
By Lemma~\ref{lem:rank-one-normalizer}, $|N|=2|T^+|=2(q-1)/|Z|$. Since $|U|=q$, we have $\gcd(|N|,|U|)=1$, and hence $N\cap U=1$. So $h\mapsto hU$ is a bijection from $N$ onto the $U$-cosets arising from $N$. Since $\alpha$ and $h\mapsto h\alpha(h)$ are permutations of $N$, the same holds for $\widetilde y$ and $\widetilde z$.

We now define the bijections on $I_1$. For $\bar a\in\Lambda$, let $t_{\bar a}$ be chosen as in Lemma~\ref{lem:rank-one-choice}, and for $x\in\F_q^\times$, set
\begin{equation*}\label{eq: lambda}
d_{\bar a}:=\frac{a^2t_{\bar a}}{1+a^2t_{\bar a}},\quad
\lambda(x,\bar a):=\frac{1}{a^2(1-d_{\bar a})x}.
\end{equation*}
Since $t_{\bar a}\neq0$ and $1+a^2t_{\bar a}\neq0$, we have $d_{\bar a}\neq0,1$, so both expressions are well defined. For $(x,\bar a)\in I_1=\F_q^\times\times\Lambda$, define
\begin{align*}
\widetilde x(x,\bar a):=B_{x,\bar a},\quad\widetilde y(x,\bar a):=B_{\lambda(x,\bar a),\,\overline{a^2d_{\bar a}(1-d_{\bar a})x^2}},\quad\widetilde z(x,\bar a):=B_{d_{\bar a}x,\,\overline{ad_{\bar a}x}}.
\end{align*}

By Lemma~\ref{lem:rank-one-cosets}, the $U$-cosets not arising from $N$ are precisely $\{B_{x,\bar a}:x\in\F_q^\times,\ \bar a\in\Lambda\}$, with unique such representations. Hence $\widetilde x$ is a bijection from $I_1$ onto this set.

To prove that $\widetilde y$ is injective, suppose that $\widetilde y(x,\bar a)=B_{\lambda(x,\bar a),\bar b}$, where $b=a^2d_{\bar a}(1-d_{\bar a})x^2$. By the uniqueness of the representation $B_{\lambda(x,\bar a),\bar b}$, $\lambda(x,\bar a)$ and $\bar b$ are uniquely determined. By the definitions of $d_{\bar a}$ and $\lambda(x,\bar a)$,
\[
b\lambda(x,\bar a)^2=\frac{d_{\bar a}}{a^2(1-d_{\bar a})}=t_{\bar a}.
\]
Hence $\overline{b\lambda(x,\bar a)^2}=\overline{t_{\bar a}}=\rho(\bar a)$, where $\rho$ is the permutation defined in Lemma \ref{lem:rank-one-choice}. Thus $\rho(\bar a)$ is determined by $B_{\lambda(x,\bar a),\bar b}$. Since $\rho$ is a permutation of $\Lambda$, $\bar a$ is uniquely determined. Consequently, $a^2$ and $d_{\bar a}$ are determined, and by the definition of $\lambda(x,\bar a)$, $x=\frac{1}{a^2(1-d_{\bar a})\lambda(x,\bar a)}$ is uniquely determined as well. Therefore $\widetilde y$ is injective, and hence bijective.

Similarly, by the uniqueness of the representation $\widetilde z(x,\bar a)=B_{d_{\bar a}x,\,\overline{ad_{\bar a}x}}$, its value determines $d_{\bar a}x$ and $\overline{ad_{\bar a}x}$. Hence $\bar a=\overline{ad_{\bar a}x}\,\overline{d_{\bar a}x}^{\, -1}$ is uniquely determined. Thus $d_{\bar a}$ is determined, and so is $x=\frac{d_{\bar a}x}{d_{\bar a}}$. Therefore $\widetilde z$ is injective, and hence bijective.

It remains to verify $\widetilde z(i)\subseteq\widetilde x(i)\widetilde y(i)$, $i\in I$. For $(x,\bar a)\in I_1$, set $\lambda:=\lambda(x,\bar a)$, $b:=a^2d_{\bar a}(1-d_{\bar a})x^2$. By the definition of $\lambda(x,\bar a)$, we have $x-\frac{1}{a^2\lambda}=d_{\bar a}x$, $ab\lambda=ad_{\bar a}x$. Hence Lemma~\ref{lem:rank-one-cosets} gives
\[
\widetilde z(x,\bar a)=B_{d_{\bar a}x,\,\overline{ad_{\bar a}x}}
\subseteq B_{x,\bar a}B_{\lambda,\,\overline{b}}=\widetilde x(x,\bar a)\widetilde y(x,\bar a).
\]
Next, let $h\in N$. Since $1\in U$, $\widetilde z(h)=h\alpha(h)U\subseteq hU\,\alpha(h)U=\widetilde x(h)\widetilde y(h)$. Thus the required inclusion holds for every $i\in I$.

On $N$, each of $\widetilde x,\widetilde y,\widetilde z$ is a bijection onto the $U$-cosets arising from $N$, while on $I_1$, each is a bijection onto the remaining $U$-cosets. Since $I=N\mathbin{\dot\cup}I_1$, $\widetilde x,\widetilde y,\widetilde z:I\to G/U$ are all bijections.
\end{proof}

It remains to construct the simultaneous left and right $U$-transversal required by Lemma~\ref{lem:local-wilcox}.

\begin{lem}\label{lem:rank-one-transversal}
Set $\mathcal T:=\{h(a):\bar a\in\Lambda\}\mathbin{\dot\cup}\{u(x)wh(a)u(x):x\in\F_q,\bar a\in\Lambda\}$. Then $\mathcal T$ is simultaneously a transversal for the left and right $U$-cosets of $G$. Moreover, $N\subseteq\mathcal T$.
\end{lem}

\begin{proof}
By Lemma~\ref{lem:rank-one-cosets}, $h(a)U=A_{\bar a}$, $u(x)wh(a)u(x)U=B_{x,\bar a}$. Hence $\mathcal T$ is a transversal for the left $U$-cosets of $G$.

Moreover, $h(a)^{-1}=h(a^{-1})$ and $\bigl(u(x)wh(a)u(x)\bigr)^{-1}=u(-x)wh(-a)u(-x)$. Thus $\mathcal T^{-1}=\mathcal T$. Since inversion interchanges left and right $U$-cosets, $\mathcal T$ is also a transversal for the right $U$-cosets.

Finally, $T^{+}\subseteq\mathcal T$, $wh(a)=u(0)wh(a)u(0)\in\mathcal T$. By Lemma~\ref{lem:rank-one-normalizer}, $N=T^+\mathbin{\dot\cup}wT^+$. Hence $N\subseteq\mathcal T$.
\end{proof}

\begin{pro}\label{pro:SL2-one-mod-four}
Assume that $q\equiv1\pmod4$. Then $\mathsf{CM}(2)$ holds.
\end{pro}

\begin{proof}
By Lemma~\ref{lem:rank-one-normalizer}, $N$ contains some $P\in\mathrm{Syl}_2(G)$ and admits a complete mapping $\alpha$ satisfying $\alpha(P)=P$. Let $\widetilde x,\widetilde y,\widetilde z:I\longrightarrow G/U$ be the bijections given by Lemma~\ref{lem:rank-one-coset-maps}, and let $\mathcal T$ be the simultaneous left and right $U$-transversal given by Lemma~\ref{lem:rank-one-transversal}.

For each $i\in I$, let $y_i$ be the unique element of $\mathcal T\cap\widetilde y(i)$. For $h\in N$, we have $\widetilde y(h)=\alpha(h)U$. Since $\alpha(h)\in N\subseteq\mathcal T$, it follows that $y_h=\alpha(h)$.

To apply Lemma~\ref{lem:local-wilcox} with $H=U$ and $S=G$, we need a complete mapping triple of $U$. Since $q$ is odd, the identity map $\mathrm{id}_U$ is a complete mapping, because $u(x)\mathrm{id}_U(u(x))=u(x)^2=u(2x)$, and $x\mapsto2x$ is a permutation of $\F_q$. Thus, taking $J=U$, we obtain the complete mapping triple $(a,b,c)$ given by $a(u)=u$, $b(u)=u$, $c(u)=u^2$. Set $j_0=1_U$. Then $a(j_0)=b(j_0)=c(j_0)=1$.

For every $h\in N$, prescribe $x_h:=h$, $z_h:=h\alpha(h)$. Then $x_h\in\widetilde x(h)$, $y_h=\alpha(h)\in\widetilde y(h)$, $z_h\in\widetilde z(h)$, and $z_h=x_hy_h$. Hence Lemma~\ref{lem:local-wilcox} yields a complete mapping triple $(A,B,C)$ on $G$, with common index set $I\times J$, such that $A(h,j_0)=h$, $B(h,j_0)=\alpha(h)$, $h\in N$.

Define $\theta:=B\circ A^{-1}$. Then $\theta$ is a complete mapping of $G$, and for every $h\in N$, $\theta(h)=B(A^{-1}(h))=B(h,j_0)=\alpha(h)$. Thus $\theta|_N=\alpha$, and so $\theta(P)=\alpha(P)=P$. Hence $\mathsf{CM}(2)$ holds.
\end{proof}

\subsection{The case $n=2^a$ with $a\geq2$}

Throughout this subsection, we retain from Subsection~\ref{subsec:equal-block-lifting} the notation $E_1,E_2,F,X,\widehat L_0,L_0,\widehat w,w,K$, and assume that $s=\dim F=0$.

We prove $\mathsf{CM}(2^a)$ for all $a\geq1$ by induction on $a$. The case $a=1$, namely $\mathsf{CM}(2)$, was proved in the preceding subsection. Assume $a\geq2$ and set $m=2^{a-1}$. Then
\[
V=E_1\oplus E_2,\quad \dim E_1=\dim E_2=m,
\]
and $X=G$. We first construct a complete mapping of $K$ stabilizing a Sylow $2$-subgroup, and then apply Proposition~\ref{pro:relative-equal-block-lifting}.

\begin{lem}\label{lem:power-two-index}
We have $[G:K]=\frac{1}{2}q^{m^2}\binom{2m}{m}_q$, which is odd.
\end{lem}

\begin{proof}
The group $G$ acts transitively on the unordered decompositions
\[
V=E_1'\oplus E_2',\quad\dim E_1'=\dim E_2'=m.
\]
Indeed, given two decompositions $V=E_1'\oplus E_2'$ and $V=E_1''\oplus E_2''$, there exists $g_0\in\GL(V)$ such that $g_0(E_i')=E_i''$ for $i=1,2$. Choose $t\in\GL(V)$ preserving both $E_1''$ and $E_2''$ and satisfying $\det(t)=\det(g_0)^{-1}$. Then $tg_0\in\SL(V)$ and $tg_0(E_i')=E_i''$, $i=1,2$. Hence $G$ acts transitively on such decompositions, and therefore also on the corresponding unordered decompositions.

By the definition of $K$ in Subsection \ref{subsec:equal-block-lifting}, $K=L_0\mathbin{\dot\cup}wL_0$, where $L_0$ fixes $E_1$ and $E_2$ individually, while $w$ interchanges them. Therefore $K=G_{\{E_1,E_2\}}$. Moreover, by \cite[Theorem~1.1.1 and Lemma~1.1.2]{GodsilFiniteGeometry}, there are $\binom{2m}{m}_q$ choices for $E_1'$, and each $E_1'$ has $q^{m^2}$ complements. Since every unordered decomposition corresponds to two ordered ones,
\[
[G:K]=\frac{1}{2}q^{m^2}\binom{2m}{m}_q.
\]
Since $q$ is odd and $m\geq2$ is a power of $2$, Lemma~\ref{lem:equal-qbinom-valuation} gives $\nu_2\left(\binom{2m}{m}_q\right)=1$. So $[G:K]$ is odd.
\end{proof}

To construct the required complete mapping of $K$, we first introduce a subgroup $M\leq L_0$ and establish the following auxiliary result.

\begin{lem}\label{lem:power-two-M}
Set $\widehat M:=\{\diag(A,B):A,B\in\SL_m(q)\}$, $M:=\pi(\widehat M)$. If $\mathsf{CM}(m)$ holds, then $M$ admits a complete mapping stabilizing any prescribed Sylow $2$-subgroup.
\end{lem}

\begin{proof}
Set
\[
\widehat M_1:=\{\diag(A,I_m):A\in\SL_m(q)\},\quad M_1:=\pi(\widehat M_1).
\]
Since $\widehat M_1\trianglelefteq\widehat M$, we have $M_1\trianglelefteq M$. Moreover, $\widehat M_1\cap Z=1$, so $M_1\cong\SL_m(q)$.

Set
\[
Z_1:=\{\lambda I_m:\lambda I_{2m}\in Z,\ \lambda^m=1\}\leq Z(\SL_m(q)).
\]
Then $\varphi:M\longrightarrow\SL_m(q)/Z_1$, $\varphi\bigl(\pi(\diag(A,B))\bigr):=BZ_1$ is a well-defined surjective homomorphism. If $\varphi(\pi(\diag(A,B)))=Z_1$, then $B=\lambda I_m$ for some $\lambda I_{2m}\in Z\cap\widehat M$, and $\pi(\diag(A,B))=\pi(\diag(\lambda^{-1}A,I_m))\in M_1$. Thus $\ker\varphi=M_1$, and hence $M/M_1\cong\SL_m(q)/Z_1$.

Let $P_M\in\mathrm{Syl}_2(M)$. Since $M_1\trianglelefteq M$,
$$
P_M\cap M_1\in\mathrm{Syl}_2(M_1),\quad P_MM_1/M_1\in\mathrm{Syl}_2(M/M_1).
$$
Since $M_1\cong\SL_m(q)$ and $M/M_1\cong\SL_m(q)/Z_1$ with $Z_1\leq Z(\SL_m(q))$, $\mathsf{CM}(m)$, together with the conjugacy of Sylow $2$-subgroups, gives complete mappings of $M_1$ and $M/M_1$ stabilizing $P_M\cap M_1$ and $P_MM_1/M_1$, respectively. Thus Lemma~\ref{lem:normal-extension} yields a complete mapping of $M$ stabilizing $P_M$.
\end{proof}

We now use $M$ to construct the required complete mapping of $K$ and to verify that $L_0$ admits a complete mapping.

\begin{lem}\label{lem:power-two-K}
Assume that $\mathsf{CM}(m)$ holds. Then $L_0$ admits a complete mapping, and $K$ admits a complete mapping stabilizing a Sylow $2$-subgroup.
\end{lem}

\begin{proof}
We first show that $M\trianglelefteq L_0$ and determine $L_0/M$. Consider the surjective homomorphism
\[
\delta:\widehat L_0\longrightarrow\F_q^\times,\quad\diag(A,B)\mapsto\det A.
\]
Clearly, $\ker\delta=\widehat M$. Set $E:=\delta(Z)=\{\lambda^m:\lambda I_{2m}\in Z\}$. Then $\delta$ induces a homomorphism
\[
\overline\delta:L_0\longrightarrow\F_q^\times/E,\quad \pi(x)\mapsto\delta(x)E,\quad x\in\widehat L_0.
\]
This map is well defined. Indeed, if $\pi(x)=\pi(y)$, then $y=xz$ for some $z\in Z$, and hence $\delta(y)E=\delta(x)E$. We claim that $\ker\overline\delta=M$. Indeed, if $\pi(x)\in\ker\overline\delta$, then $\delta(x)\in E=\delta(Z)$, so $\delta(x)=\delta(z)$ for some $z\in Z$. Hence $xz^{-1}\in\ker\delta=\widehat M$, and therefore $\pi(x)=\pi(xz^{-1})\in M$. The reverse inclusion is immediate. Thus $\ker\overline\delta=M$, so $M\trianglelefteq L_0$, and
\[
L_0/M\cong\F_q^\times/E.
\]
Let $d:=|\F_q^\times/E|$. We claim that $d$ is even. If $\lambda I_{2m}\in Z$, then $1=\det(\lambda I_{2m})=\lambda^{2m}$, so $(\lambda^m)^2=1$. Since $q$ is odd, $\lambda^m\in\{\pm1\}$, and hence $E\leq\{\pm1\}$. If $E=1$, then $d=q-1$ is even. If $E=\{\pm1\}$, there exists $\lambda I_{2m}\in Z$ such that $\lambda^m=-1$. Since $m$ is a power of $2$, $\lambda$ has order $2m$, so $4\mid 2m\mid q-1$. Thus $d=(q-1)/2$ is again even. Hence $L_0/M$ is cyclic of even order $d$.

We next show that $M\trianglelefteq K$ and determine $K/M$. Since $\widehat w$ interchanges $E_1$ and $E_2$, for $\diag(A,B)\in\widehat M$ we have
\[
\widehat w\diag(A,B)\widehat w^{-1}=\diag(B,A)\in\widehat M.
\]
Thus $w$ normalizes $M$. Since $M\trianglelefteq L_0$ and $K=L_0\mathbin{\dot\cup}wL_0$, it follows that $M\trianglelefteq K$. Moreover,
\[
K/M=(L_0/M)\mathbin{\dot\cup}(wM)(L_0/M)=\langle L_0/M,wM\rangle.
\]
For $g=\pi(\diag(A,B))\in L_0$, we have $wgw^{-1}=\pi(\diag(B,A))$. Since $\det A\det B=1$,
\[
\overline\delta(wgw^{-1})=\overline\delta(\pi(\diag(B,A)))=\det(B)E=(\det A)^{-1}E=((\det A)E)^{-1}=\overline\delta(g)^{-1}.
\]
It follows that $(wM)(gM)(wM)^{-1}=g^{-1}M$. Thus conjugation by $wM$ acts as inversion on $L_0/M$. Also, $w^2\in M$ and $w\notin L_0$, so $wM$ has order $2$. Recall that $L_0/M$ is cyclic of order $d$. Hence
\[
K/M\cong C_d\rtimes C_2\cong D_{2d}.
\]

We now construct a complete mapping of $K$ stabilizing a Sylow $2$-subgroup. Fix $P\in\mathrm{Syl}_2(K)$ and set $\overline P:=PM/M\in\mathrm{Syl}_2(K/M)$. Since $K/M\cong D_{2d}$, $\overline P$ is noncyclic.

Write $d=2^b d'$, where $b\geq1$ and $d'$ is odd, and let $R\leq L_0/M$ have order $d'$. Since $|K/M|=2d=2^{b+1}d'$, we have $|\overline P|=2^{b+1}$. As $|R|=d'$ is odd, $R\cap\overline P=1$, and hence $|R\overline P|=|R||\overline P|=|K/M|$. Therefore
\[
K/M=R\overline P,\quad R\cap\overline P=1.
\]
The group $R$ has odd order, while $\overline P$ is a noncyclic $2$-group. Hence both admit complete mappings, and Lemma \ref{lem:exact-factorization} gives a complete mapping of $K/M$ stabilizing $\overline P$. Since $M\trianglelefteq K$ and $P\in\mathrm{Syl}_2(K)$, $P\cap M\in\mathrm{Syl}_2(M)$. By Lemma~\ref{lem:power-two-M}, $M$ admits a complete mapping stabilizing $P\cap M$. Lemma~\ref{lem:normal-extension} therefore yields a complete mapping of $K$ stabilizing $P$.

It remains to show that $L_0$ admits a complete mapping. By Lemma~\ref{lem:power-two-M}, $M$ admits a complete mapping. Moreover, $\pi(\diag(-I_m,I_m))$ is a nontrivial involution of $M$, so $|M|$ is even. Hence, by Theorem \ref{thm:HP}, any $Q\in\mathrm{Syl}_2(M)$ is noncyclic. Choose $P_0\in\mathrm{Syl}_2(L_0)$ with $Q\leq P_0$. Then $P_0$ is noncyclic, and hence Theorem \ref{thm:HP} implies that $L_0$ admits a complete mapping.
\end{proof}

\begin{pro}\label{pro:power-two}
For every $a\geq1$, $\mathsf{CM}(2^a)$ holds.
\end{pro}

\begin{proof}
We argue by induction on $a$. The case $a=1$, namely $\mathsf{CM}(2)$, follows from Proposition \ref{prop:n2-q3mod4} and Proposition~\ref{pro:SL2-one-mod-four}. For $a\geq2$, set $m=2^{a-1}$ and assume that $\mathsf{CM}(m)$ holds.

By Lemma~\ref{lem:power-two-K}, $K$ has a complete mapping $\alpha$ stabilizing some $P\in\mathrm{Syl}_2(K)$, while $L_0$ admits a complete mapping. Hence Proposition~\ref{pro:relative-equal-block-lifting} yields a complete mapping $\theta$ of $X=G$ such that $\theta|_K=\alpha$. By Lemma~\ref{lem:power-two-index}, $[G:K]$ is odd, so $P\in\mathrm{Syl}_2(G)$. Therefore $\theta(P)=\alpha(P)=P$. Thus $\mathsf{CM}(2m)$ holds, completing the induction.
\end{proof}

\subsection{The case $n=3$ and $q\equiv3\pmod4$}

We now turn to $\mathsf{CM}(3)$. We treat here the case $n=3$ and $q\equiv3\pmod4$. The case $n=3$ with $q\equiv1\pmod4$ will be covered in Subsection \ref{subsec:general-dimensions} as part of the general argument. Thus $V=\F_q^3$ and $G=\SL_3(q)/Z$, where $Z\leq Z(\SL_3(q))$. Fix a decomposition
\[
V=E\oplus W,\quad \dim E=1,\quad \dim W=2.
\]
Then we have two flags
\[
\mathcal F:0<V, \quad \mathcal F':0<E<V.
\]
Indeed, $\mathcal F'$ is obtained from $\mathcal F$ by inserting the one-dimensional subspace $E$. Clearly, $G_{\mathcal F}=G$ and $G_{\mathcal F'}=G_E$.

Denote $Q:=G_{\mathcal F'}=G_E$. We will show that $[G:Q]$ is odd and construct a complete mapping of $Q$ stabilizing a Sylow $2$-subgroup. Once this is done, Proposition \ref{prop:one-step-flag-lifting} applies to the pair of flags $\mathcal F$ and $\mathcal F'$, and shows that any complete mapping of $Q$ stabilizing a Sylow $2$-subgroup extends to a complete mapping of $G$ stabilizing the same Sylow $2$-subgroup.

Since $Q=G_{\mathcal F'}$, its preimage $\widehat Q:=\pi^{-1}(Q)$ under the natural quotient map $\pi:\SL(V)\to G$ is precisely the stabilizer of $\mathcal F'$ in $\SL(V)$. By Lemma~\ref{lem:flag-Levi}, $\widehat Q$ admits a decomposition
\begin{equation}\label{eq:widehat Q}
\widehat Q=\widehat U_Q\rtimes\widehat L_Q,
\end{equation}
where
\begin{equation}\label{eq: widehat U_Q widehat L_Q}
\widehat U_Q=
\left\{
\begin{pmatrix}
1&u\\0&I_2
\end{pmatrix}:u\in\mathrm{Hom}(W,E)
\right\},\quad
\widehat L_Q=
\left\{
\begin{pmatrix}
(\det A)^{-1}&0\\0&A
\end{pmatrix}:A\in\GL(W)
\right\}.
\end{equation}
Set $U_Q:=\pi(\widehat U_Q)$, $L_Q:=\pi(\widehat L_Q)$. Since $\widehat U_Q\trianglelefteq \widehat Q$, $U_Q\trianglelefteq Q$. Moreover, the decomposition \eqref{eq:widehat Q} gives
\[
Q=U_Q\rtimes L_Q,\quad Q/U_Q\cong L_Q .
\]
The following lemma describes $U_Q$ and $L_Q$ explicitly in the present case.

\begin{lem}\label{lem:n3-parabolic}
With the notation above, $[G:Q]=q^2+q+1$, which is odd. Moreover, $|U_Q|=q^2$ and $L_Q\cong\GL_2(q)/C$, where $C:=\{\lambda I_2:\lambda I_3\in Z\}$ is a central subgroup of $\GL_2(q)$ of odd order.
\end{lem}

\begin{proof}
Clearly, $G$ acts transitively on the $1$-dimensional subspaces of $V$, with stabilizer $Q=G_E$. Hence, by the orbit-stabilizer theorem, $[G:Q]=\binom{3}{1}_q=q^2+q+1$, which is odd.

By \eqref{eq: widehat U_Q widehat L_Q}, we have $|\widehat U_Q|=|\operatorname{Hom}(W,E)|=q^2$. Moreover, every element of $Z$ is scalar, whereas the only scalar matrix in $\widehat U_Q$ is the identity. Hence $\widehat U_Q\cap Z=1$, and therefore $|U_Q|=q^2$.

We next determine $L_Q$. Define
\begin{equation}\label{eq: psi}
\psi:\GL(W)\longrightarrow L_Q,\quad \psi(A) :=\pi\left(\begin{pmatrix}
(\det A)^{-1}&0\\0&A
\end{pmatrix}\right).
\end{equation}
By the definition of $\widehat L_Q$, the map $\psi$ is surjective. We claim that $\ker\psi=C$. Indeed, if $A\in\ker\psi$, then $\operatorname{diag}((\det A)^{-1},A)\in Z$. Thus, for some $\lambda\in\F_q^\times$ with $\lambda I_3\in Z$, $A=\lambda I_2$. Hence $A\in C$. Conversely, if $A=\lambda I_2\in C$, then $\lambda I_3\in Z$, and hence $\lambda^3=1$. Therefore $(\det A)^{-1}=\lambda^{-2}=\lambda$, so $\operatorname{diag}((\det A)^{-1},A)=\lambda I_3\in Z$. Thus $A\in\ker\psi$, proving $\ker\psi=C$. Consequently,
\[
L_Q\cong \GL(W)/C\cong \GL_2(q)/C.
\]

Finally, $C\cong\{\lambda\in\F_q^\times:\lambda I_3\in Z\}\leq \{\lambda\in\F_q^\times:\lambda^3=1\}$. Hence $|C|\mid 3$, so $C$ has odd order.
\end{proof}

We next construct a complete mapping of the factor $L_Q$ stabilizing a Sylow $2$-subgroup.

\begin{lem}\label{lem:LQ-complete-mapping}
The group $L_Q$ admits a complete mapping stabilizing a Sylow $2$-subgroup.
\end{lem}

\begin{proof}
Let $\psi:\GL_2(q)\to L_Q$ be the surjective homomorphism defined in \eqref{eq: psi}. By Lemma \ref{lem:n3-parabolic}, its kernel is the odd-order subgroup $C$.

We first construct a subgroup $N\leq L_Q$ together with $P\in\mathrm{Syl}_2(N)$ such that $N$ admits a complete mapping stabilizing $P$. We will later show that $[L_Q:N]$ is odd, and hence $P\in\mathrm{Syl}_2(L_Q)$. We then find an odd-order subgroup $H\leq L_Q$ such that $L_Q=HN$ is an exact factorization. Lemma~\ref{lem:exact-factorization} then allows us to extend the complete mapping from $N$ to $L_Q$ while preserving its restriction to $N$.

Let $\widetilde T^{-}$ be the nonsplit maximal torus of $\GL_2(q)$ defined in Definition \ref{def:nonsplit-torus}. Then $\widetilde T^{-}\cong\F_{q^2}^{\times}\cong C_{q^2-1}$. Set
\[
\widetilde s=
\begin{pmatrix}
1&0\\0&-1
\end{pmatrix}.
\]
For
\[
t=
\begin{pmatrix}
a&b\\ \varepsilon b&a
\end{pmatrix}\in\widetilde T^{-},
\]
we have
\[
\widetilde s t\widetilde s^{-1}=
\begin{pmatrix}
a&-b\\-\varepsilon b&a
\end{pmatrix}
\in\widetilde T^{-}.
\]
Thus $\widetilde s$ normalizes $\widetilde T^{-}$. Since $\widetilde s\notin\widetilde T^{-}$ and $\widetilde s^2=1$, we may set
\[
\widetilde N:=\widetilde T^{-}\rtimes\langle\widetilde s\rangle=\widetilde T^{-}\mathbin{\dot\cup}\widetilde s\widetilde T^{-},
\]
and hence $|\widetilde N|=2(q^2-1)$. Let $\widetilde T^{-}_2$ be the Sylow $2$-subgroup of $\widetilde T^{-}$, and let $\widetilde R$ be the unique largest odd-order subgroup of $\widetilde T^{-}$. Since $\widetilde T^{-}$ is cyclic, $\widetilde T^{-}=\widetilde R\times\widetilde T^{-}_2$. Both $\widetilde R$ and $\widetilde T^{-}_2$ are characteristic in $\widetilde T^{-}$ and hence are normalized by $\widetilde s$. Therefore $\widetilde P:=\widetilde T^{-}_2\rtimes\langle\widetilde s\rangle\in\mathrm{Syl}_2(\widetilde N)$, and
\[
\widetilde N=\widetilde R\widetilde P,\quad\widetilde R\cap\widetilde P=1.
\]
Moreover, $\widetilde P$ is noncyclic, since $\widetilde s$ is an involution not contained in the cyclic group $\widetilde T^{-}_2$.

Since every element of $C$ is scalar, $C\leq\widetilde T^{-}$. Since $|C|$ is odd and $\widetilde R$ is the unique largest odd-order subgroup of $\widetilde T^{-}$, it follows that $C\leq\widetilde R$. Set
\[
N:=\psi(\widetilde N),\quad R:=\psi(\widetilde R),\quad P:=\psi(\widetilde P).
\]
Since $C$ has odd order whereas $\widetilde P$ is a $2$-group, $C\cap\widetilde P=1$. Hence $P=\psi(\widetilde P)\cong\widetilde P/(\widetilde P\cap C)=\widetilde P$ is a noncyclic $2$-group. Moreover, $|R|$ is odd, and so $R\cap P=1$. Finally, since $\widetilde N=\widetilde R\widetilde P$ and $\psi$ is a homomorphism,
\[
N=\psi(\widetilde N)=\psi(\widetilde R)\psi(\widetilde P)=RP.
\]
Thus $N=RP$ is an exact factorization. Since $R$ has odd order and $P$ is a noncyclic $2$-group, we have $P\in\mathrm{Syl}_2(N)$ and both $R$ and $P$ admit complete mappings. Therefore, by Lemma~\ref{lem:exact-factorization}, any prescribed complete mapping of $P$ extends to a complete mapping $\alpha:N\longrightarrow N$ satisfying $\alpha(P)=P$.

It remains to extend $\alpha$ from $N$ to $L_Q$. For this purpose, let $D:=(\F_q^\times)^2$ be the subgroup of squares in $\F_q^\times$, and define
\[
\widetilde H:=
\left\{
\begin{pmatrix}
d&x\\0&1
\end{pmatrix}
:d\in D,\ x\in\F_q\right\}\leq \GL_2(q).
\]
Since $q\equiv3\pmod4$, $|\widetilde H|=\frac{q(q-1)}{2}$ is odd.

We claim that $\widetilde H\cap\widetilde N=1$. Since $\widetilde N=\widetilde T^{-}\mathbin{\dot\cup}\widetilde s\widetilde T^{-}$, it suffices to consider the two parts separately. If
\[
\begin{pmatrix}
d&x\\0&1
\end{pmatrix}=
\begin{pmatrix}
a&b\\\varepsilon b&a
\end{pmatrix},
\]
then $b=x=0$, $a=d=1$. Hence $\widetilde H\cap\widetilde T^{-}=1$. On the other hand, every element of $\widetilde s\widetilde T^{-}$ has the form
\[
\begin{pmatrix}
a&b\\
-\varepsilon b&-a
\end{pmatrix}.
\]
If such a matrix belonged to $\widetilde H$, then $b=0$ and $a=-1$. But $-1$ is not a square in $\F_q^\times$ because $q\equiv3\pmod4$. Thus $\widetilde H\cap\widetilde s\widetilde T^{-}=\varnothing$. Therefore $\widetilde H\cap\widetilde N=\widetilde H\cap\widetilde T^{-}=1$.

Finally, $|\widetilde H|\,|\widetilde N|=\frac{q(q-1)}2\cdot2(q^2-1)=|\GL_2(q)|$. Hence $\GL_2(q)=\widetilde H\widetilde N$ is an exact factorization. Set $H:=\psi(\widetilde H)$. Since $\psi$ is surjective,
\[
L_Q=\psi(\GL_2(q))=\psi(\widetilde H\widetilde N)=\psi(\widetilde H)\psi(\widetilde N)=HN.
\]
Moreover, if $x\in H\cap N$, then $x=\psi(h)=\psi(n)$ for some $h\in\widetilde H$ and $n\in\widetilde N$. Hence $hn^{-1}\in\ker\psi=C\leq\widetilde N$, so $h\in\widetilde N$. Since $h\in\widetilde H$ and $\widetilde H\cap\widetilde N=1$, $h=1$, and hence $x=1$. Therefore $H\cap N=1$, and $L_Q=HN$ is an exact factorization.

Since $\widetilde H$ has odd order, so does its homomorphic image $H$. Hence $H$ admits a complete mapping. By Lemma \ref{lem:exact-factorization}, the complete mapping $\alpha$ of $N$ extends to a complete mapping $\theta$ of $L_Q$ with $\theta|_N=\alpha$. Moreover, since $L_Q=HN$ is an exact factorization, $[L_Q:N]=|H|$, which is odd. Since $P\in\mathrm{Syl}_2(N)$, it follows that $P\in\mathrm{Syl}_2(L_Q)$. Finally, $\theta(P)=\alpha(P)=P$. Thus $L_Q$ admits a complete mapping stabilizing a Sylow $2$-subgroup.
\end{proof}

\begin{pro}\label{prop:CM3-q3}
Assume that $q\equiv3\pmod4$. Then $\mathsf{CM}(3)$ holds.
\end{pro}

\begin{proof}
By Lemma~\ref{lem:LQ-complete-mapping}, there exist $P\in\mathrm{Syl}_2(L_Q)$ and a complete mapping $\alpha$ of $L_Q$ stabilizing $P$. Since $Q=U_Q\rtimes L_Q$, the product $Q=U_QL_Q$ is an exact factorization. Moreover, by Lemma~\ref{lem:n3-parabolic}, $|U_Q|=q^2$, which is odd. Hence $U_Q$ admits a complete mapping. Therefore Lemma~\ref{lem:exact-factorization} extends $\alpha$ to a complete mapping $\theta_Q$ of $Q$. Since $\theta_Q|_{L_Q}=\alpha$, we have $\theta_Q(P)=P$. Also, $[Q:L_Q]=|U_Q|$ is odd, so $P\in\mathrm{Syl}_2(Q)$.

By Lemma~\ref{lem:n3-parabolic}, $[G:Q]=q^2+q+1$ is odd. Therefore, applying Proposition~\ref{prop:one-step-flag-lifting} to the refinement $\mathcal F':0<E<V$ of $\mathcal F:0<V$, the complete mapping $\theta_Q$ extends to a complete mapping $\theta$ of $G$ such that $\theta|_Q=\theta_Q$. Moreover, $P\in\mathrm{Syl}_2(G)$ and $\theta(P)=\theta_Q(P)=P$. Thus $\mathsf{CM}(3)$ holds.
\end{proof}

\subsection{The general case}\label{subsec:general-dimensions}

We now prove $\mathsf{CM}(n)$ for the remaining values of $n$. Write the binary expansion of $n$ as
\[
n=d_1+\cdots+d_r,\quad d_1>\cdots>d_r,
\]
where each $d_i$ is a power of $2$. The case $r=1$ was proved in Proposition \ref{pro:power-two}. We therefore assume throughout this subsection that $r\geq2$.

We first establish two auxiliary results for block-diagonal subgroups associated with decompositions of $V$ into subspaces of $2$-power dimensions.

\begin{lem}\label{lem:block-SL}
Let $V=E_1\oplus\cdots\oplus E_t$, where each $e_i:=\dim E_i$ is a power of $2$. Set
\[
\widehat M:=\left\{\mathrm{diag}(A_1,\ldots,A_t):A_i\in\SL(E_i)\right\},\quad M:=\pi(\widehat M).
\]
Then $M$ admits a complete mapping stabilizing any prescribed Sylow $2$-subgroup.
\end{lem}

\begin{proof}
Set $\widehat M_0:=1$, $M_0:=1$, and, for $1\leq j\leq t$, let
\[
\widehat M_j:=\left\{\mathrm{diag}(A_1,\ldots,A_j,I_{e_{j+1}},\ldots,I_{e_t}): A_i\in\SL(E_i)\right\},\quad M_j:=\pi(\widehat M_j).
\]
For each $1\leq j\leq t$, the homomorphism $\widehat M_j\longrightarrow\SL(E_j)$, $\mathrm{diag}(A_1,\ldots,A_j,I,\ldots,I)\mapsto A_j$ has kernel $\widehat M_{j-1}$. Hence $\widehat M_{j-1}\trianglelefteq\widehat M_j$, and therefore $M_{j-1}=\pi(\widehat M_{j-1})\trianglelefteq\pi(\widehat M_j)=M_j$. Thus
\[
1=M_0\trianglelefteq M_1\trianglelefteq\cdots\trianglelefteq M_t=M.
\]
Moreover, under the identification $\widehat M\cong \SL(E_1)\times\cdots\times\SL(E_t)$, the subgroup $\widehat M_j$ is the product of the first $j$ direct factors of $\widehat M$, so $\widehat M_j\trianglelefteq\widehat M$. Hence $M_j\trianglelefteq M$, $0\leq j\leq t$.

For each $1\leq j\leq t$, define
\[
\varphi_j:\SL(E_j)\longrightarrow M_j/M_{j-1},\quad \varphi_j(A)=\pi\!\left(\mathrm{diag}(I_{e_1},\ldots,I_{e_{j-1}},A,I_{e_{j+1}},\ldots,I_{e_t})\right)M_{j-1}.
\]
Then $\varphi_j$ is surjective. Moreover, if $A\in\ker\varphi_j$, then $A=\lambda I_{e_j}$ for some $\lambda\in\mathbb F_q^\times$. Hence $\ker\varphi_j\leq Z(\SL(E_j))$. Therefore $M_j/M_{j-1}\cong\SL(E_j)/\ker\varphi_j$, so $M_j/M_{j-1}$ is a central quotient of $\SL(E_j)$. For each $j$, if $e_j=1$, then $\SL(E_j)=1$, so the required assertion is immediate. If $e_j\geq2$, then $e_j=2^{a_j}$ for some $a_j\geq1$, and Proposition~\ref{pro:power-two} applies.

Now fix $P\in\mathrm{Syl}_2(M)$ and let $P_j:=P\cap M_j$, $0\leq j\leq t$. Since $M_j\trianglelefteq M$, we have $P_j\in\mathrm{Syl}_2(M_j)$. Consequently, $P_jM_{j-1}/M_{j-1}\in\mathrm{Syl}_2(M_j/M_{j-1})$. By the conjugacy of Sylow $2$-subgroups, we may choose a complete mapping of $M_j/M_{j-1}$ stabilizing $P_jM_{j-1}/M_{j-1}$.

Starting with $M_0=1$, apply Lemma~\ref{lem:normal-extension} successively to $M_{j-1}\trianglelefteq M_j$, $1\leq j\leq t$. Since $P_j\cap M_{j-1}=P_{j-1}$, at the $j$-th step we obtain a complete mapping of $M_j$ stabilizing $P_j$. In particular, $M=M_t$ admits a complete mapping stabilizing $P_t=P$.
\end{proof}

\begin{lem}\label{lem:block-Levi}
Let $V=E_1\oplus\cdots\oplus E_t$, $t\geq3$, where each $e_i:=\dim E_i$ is a power of $2$. Assume that $2^{\nu_2(n)}\mid e_i$, $1\leq i\leq t$, where $n=\dim V$. Set
\[
\widehat L:=\left\{\mathrm{diag}(A_1,\ldots,A_t):A_i\in\GL(E_i),\ \prod_{i=1}^t\det A_i=1\right\},\quad L:=\pi(\widehat L).
\]
Then $L$ admits a complete mapping stabilizing any prescribed Sylow $2$-subgroup.
\end{lem}

\begin{proof}
Let $\widehat M$ and $M$ be as in Lemma~\ref{lem:block-SL}. Since $\widehat M\trianglelefteq\widehat L$, $M=\pi(\widehat M)\trianglelefteq\pi(\widehat L)=L$. Set $D:=\left\{(a_1,\ldots,a_t)\in(\F_q^\times)^t:a_1\cdots a_t=1\right\}$. Clearly, $D$ is a group with identity $(1,\ldots,1)$. Define
\[
\delta:\widehat L\longrightarrow D,\quad\operatorname{diag}(A_1,\ldots,A_t)\mapsto(\det A_1,\ldots,\det A_t).
\]
Then $\delta$ is surjective and $\ker\delta=\widehat M$.
Since $Z\leq\widehat L$, it follows that
\[
L/M\cong\widehat L/\widehat M Z\cong D/\delta(Z).
\]

We claim that $|\delta(Z)|$ is odd. Let $\lambda I_n\in Z$, and write the order of $\lambda$ as $2^c u$, where $u$ is odd. Since $\lambda^n=1$, we have $|\lambda|\mid n$, and hence $2^c\mid 2^{\nu_2(n)}$. By hypothesis, $2^{\nu_2(n)}\mid e_i$ for every $i$, so $2^c\mid e_i$. Therefore $|\lambda^{e_i}|=\frac{|\lambda|}{\gcd(|\lambda|,e_i)}$ is odd for every $i$. Hence
\[
\delta(\lambda I_n)=\delta(\diag(\lambda I_{e_1},\cdots,\lambda I_{e_t}))=(\lambda^{e_1},\ldots,\lambda^{e_t})
\]
has odd order. Thus every element of $\delta(Z)$ has odd order, and consequently $|\delta(Z)|$ is odd.

Now $D\cong(\F_q^\times)^{t-1}$. Write $c_0:=\nu_2(q-1)$. Since $|\delta(Z)|$ is odd, every Sylow $2$-subgroup of $L/M\cong D/\delta(Z)$ is isomorphic to $C_{2^{c_0}}^{\,t-1}$. As $t\geq3$, such a Sylow $2$-subgroup is noncyclic. Now fix $P\in\mathrm{Syl}_2(L)$ and let $\overline P:=PM/M$. Then
\[
P\cap M\in\mathrm{Syl}_2(M),\quad \overline P\in\mathrm{Syl}_2(L/M).
\]
Let $\overline R$ be the largest odd-order subgroup of the abelian group $L/M$. Then $L/M=\overline R\times\overline P$. The group $\overline R$ has odd order, while $\overline P$ is a noncyclic $2$-group. Hence both admit complete mappings. By Lemma~\ref{lem:exact-factorization}, $L/M$ admits a complete mapping stabilizing $\overline P$.

By Lemma~\ref{lem:block-SL}, $M$ admits a complete mapping stabilizing $P\cap M$. Therefore Lemma~\ref{lem:normal-extension} yields a complete mapping of $L$ stabilizing $P$.
\end{proof}

We now consider the case where the binary expansion of $n$ has at least three terms.

\begin{pro}\label{prop:general-r-ge-3}
If $r\geq3$, then $\mathsf{CM}(n)$ holds.
\end{pro}

\begin{proof}
Choose a decomposition $V=E_1\oplus\cdots\oplus E_r$, $\dim E_i=d_i$, and set $V_j:=E_1\oplus\cdots\oplus E_j$, $1\leq j\leq r-1$. Consider the flag
\[
\mathcal F: 0<V_1<\cdots<V_{r-1}<V,
\]
and let $Q:=G_{\mathcal F}$. By Lemma~\ref{lem:flag-Levi}, we may write $Q=U\rtimes L$, where $|U|=q^{\sum_{1\leq i<j\leq r}d_id_j}$, since the $(i,j)$-block has $q^{d_id_j}$ choices for each $i<j$. Thus $U$ has odd order. Moreover,
\[
L=\pi\left(\left\{\mathrm{diag}(A_1,\ldots,A_r):A_i\in\GL(E_i),\prod_{i=1}^r\det A_i=1\right\}\right).
\]

Since $d_r$ is the smallest term in the binary expansion of $n$, $d_r=2^{\nu_2(n)}$, and $d_r\mid d_i$ for every $i$. Hence the hypotheses of Lemma~\ref{lem:block-Levi} are satisfied. Therefore $L$ admits a complete mapping stabilizing any prescribed Sylow $2$-subgroup. Fix $P\in\mathrm{Syl}_2(L)$, and let $\alpha$ be such a complete mapping of $L$ stabilizing $P$. Since $Q=U\rtimes L$ and $|U|$ is odd, $Q=UL$ is an exact factorization and $U$ admits a complete mapping. Hence Lemma \ref{lem:exact-factorization} extends $\alpha$ to a complete mapping $\theta_Q$ of $Q$. Moreover, $[Q:L]=|U|$ is odd, so $P\in\mathrm{Syl}_2(Q)$, and $\theta_Q(P)=\alpha(P)=P$.

It remains to lift this complete mapping from $Q$ to $G$. Let $\mathcal F_0:0<V$ be the trivial flag. The flag $\mathcal F$ is obtained from $\mathcal F_0$ by inserting subspaces $V_1,\ldots,V_{r-1}$. Since $\SL(V)$ acts transitively on the flags having the same dimensions as $\mathcal F$, and $Z$ fixes every subspace of $V$, the induced action of $G$ on such flags is transitive. Since $Q=G_{\mathcal F}$, the orbit--stabilizer theorem, together with successive choices of the subspaces in such a flag, gives
\[
[G:Q]=\prod_{k=1}^{r-1}\binom{d_k+\cdots+d_r}{d_k}_q.
\]
For each $1\leq k\leq r-1$, the integers $d_k,d_{k+1},\ldots,d_r$ are distinct powers of $2$. Hence the binary addition of $d_k$ and $d_{k+1}+\cdots+d_r$ involves no carries. Therefore Lemma \ref{lem:qbinom-parity}-(2) shows that $\binom{d_k+\cdots+d_r}{d_k}_q$ is odd for every $k$. Thus $[G:Q]$ is odd.

Applying Corollary~\ref{cor:flag-refinement-lifting} to the refinement $\mathcal F$ of $\mathcal F_0$, the complete mapping $\theta_Q$ extends to a complete mapping $\theta$ of $G$ stabilizing $P\in\mathrm{Syl}_2(G)$. Hence $\mathsf{CM}(n)$ holds.
\end{proof}

It remains to consider the case $r=2$. We first establish the complete mappings needed to apply Proposition \ref{pro:relative-equal-block-lifting}.

\begin{lem}\label{lem:three-block-K}
Let
\[
V=E_1\oplus E_2\oplus F,\quad\dim E_1=\dim E_2=m,\quad\dim F=s,
\]
and let $\widehat L_0$, $L_0$, $K$, and $w$ be as defined in Subsection \ref{subsec:equal-block-lifting}. Suppose that $m$ and $s$ are powers of $2$ and that $s\mid m$. Then $L_0$ admits a complete mapping, and $K$ admits a complete mapping stabilizing any prescribed Sylow $2$-subgroup.
\end{lem}

\begin{proof}
Since $n=2m+s=s\left(\frac{2m}{s}+1\right)$ and $\frac m s$ is a power of $2$, $\frac{2m}{s}+1$ is odd. Hence $s=2^{\nu_2(n)}$. Thus the dimensions $m,m,s$ satisfy the hypotheses of Lemma \ref{lem:block-Levi}. Therefore $L_0$ admits a complete mapping.

Let
\[
\widehat M:=\left\{\mathrm{diag}(A_1,A_2,B):A_1\in\SL(E_1),\ A_2\in\SL(E_2),\ B\in\SL(F)\right\},\quad M:=\pi(\widehat M).
\]
Since $L_0$ normalizes each of the three factors of $M$, while conjugation by $w$ interchanges the first two factors and preserves the third, we have $M\trianglelefteq K$. By Lemma~\ref{lem:block-SL}, $M$ admits a complete mapping stabilizing any prescribed Sylow $2$-subgroup.

Now fix $P\in\mathrm{Syl}_2(K)$ and let $\overline P:=PM/M$. Then $\overline P\in\mathrm{Syl}_2(K/M)$. Moreover, $L_0\trianglelefteq K$ and $[K:L_0]=2$, so
\[
L_0/M\trianglelefteq K/M,\quad[K/M:L_0/M]=2.
\]
As in the proof of Lemma~\ref{lem:block-Levi}, the determinant homomorphism
\[
\widehat L_0\longrightarrow \left\{(a_1,a_2,a_3)\in(\F_q^\times)^3:a_1a_2a_3=1\right\},\quad \mathrm{diag}(A_1,A_2,B)\mapsto (\det A_1,\det A_2,\det B)
\]
shows that $L_0/M$ is a quotient of the abelian group $\left\{(a_1,a_2,a_3)\in(\F_q^\times)^3:a_1a_2a_3=1\right\}\cong(\F_q^\times)^2$. Hence $L_0/M$ is abelian. Moreover, as in the proof of Lemma~\ref{lem:block-Levi}, the subgroup arising from $Z$ under the determinant homomorphism has odd order. Therefore, writing $c_0:=\nu_2(q-1)$, every Sylow $2$-subgroup of $L_0/M$ is isomorphic to $C_{2^{c_0}}^2$. Since $L_0/M\trianglelefteq K/M$, $\overline P\cap(L_0/M)\in\mathrm{Syl}_2(L_0/M)$. In particular, $\overline P$ is noncyclic.

Let $\overline R$ be the largest odd-order subgroup of the abelian group $L_0/M$. Then $\overline R$ is characteristic in $L_0/M$. Since $L_0/M\trianglelefteq K/M$, it follows that $\overline R\trianglelefteq K/M$. Furthermore, $(L_0/M)/\overline R$ is a $2$-group. Since $[K/M:L_0/M]=2$, it follows that $(K/M)/\overline R$ is also a $2$-group. As $\overline P\in\mathrm{Syl}_2(K/M)$ and $\overline R$ has odd order, we obtain
\[
K/M=\overline R\,\overline P,\quad\overline R\cap\overline P=1.
\]
Thus $K/M=\overline R\,\overline P$ is an exact factorization. The group $\overline R$ has odd order, while $\overline P$ is a noncyclic $2$-group, so both admit complete mappings. Hence Lemma~\ref{lem:exact-factorization} yields a complete mapping of $K/M$ stabilizing $\overline P$.

Finally, since $M\trianglelefteq K$, $P\cap M\in\mathrm{Syl}_2(M)$. By Lemma~\ref{lem:block-SL}, $M$ admits a complete mapping stabilizing $P\cap M$. Together with the complete mapping of $K/M$ stabilizing $PM/M=\overline P$, Lemma~\ref{lem:normal-extension} yields a complete mapping of $K$ stabilizing $P$.
\end{proof}

We can now settle the case $r=2$.

\begin{pro}\label{prop:general-r-2}
If $r=2$, then $\mathsf{CM}(n)$ holds.
\end{pro}

\begin{proof}
Write $n=d_1+d_2$, $d_1>d_2$, and put $d_1=2m$, $s:=d_2$. Since $d_1$ and $d_2$ are powers of $2$, both $m$ and $s$ are powers of $2$, and $s\mid m$.

If $m=1$, then $d_1=2$, $s=1$, and hence $n=3$. The case $q\equiv3\pmod4$ has already been settled in Proposition~\ref{prop:CM3-q3}. Thus, when $m=1$, we may assume that $q\equiv1\pmod4$.

Choose a decomposition
\[
V=E_1\oplus E_2\oplus F,\quad\dim E_1=\dim E_2=m,\quad\dim F=s,
\]
and retain the notation $X,L_0,K$ of Subsection~\ref{subsec:equal-block-lifting}. By Lemma~\ref{lem:three-block-K}, $L_0$ admits a complete mapping, and $K$ admits a complete mapping stabilizing any prescribed Sylow $2$-subgroup. Fix $P\in\mathrm{Syl}_2(K)$, and choose a complete mapping $\alpha$ of $K$ stabilizing $P$.

By the same counting argument as in the proof of Lemma \ref{lem:power-two-index}, applied to $E_1\oplus E_2$, we have
\[
[X:K]=\frac{1}{2}q^{m^2}\binom{2m}{m}_q.
\]
If $m\geq2$, Lemma~\ref{lem:equal-qbinom-valuation}-(2) gives $\nu_2\left(\binom{2m}{m}_q\right)=1$, and hence $[X:K]$ is odd. If $m=1$, then $[X:K]=\frac{q(q+1)}{2}$, which is odd since $q\equiv1\pmod4$. Thus $[X:K]$ is odd in all cases under consideration, and therefore $P\in\mathrm{Syl}_2(X)$.

By Proposition~\ref{pro:relative-equal-block-lifting}, the complete mapping $\alpha$ of $K$ extends to a complete mapping $\theta_X$ of $X$. Since $\theta_X|_K=\alpha$, we have $\theta_X(P)=P$. Set $W:=E_1\oplus E_2$ and let $Q:=G_W$.

With respect to the decomposition $V=W\oplus F$, Lemma \ref{lem:flag-Levi} gives $Q=U\rtimes X$, where $X$ is precisely the subgroup defined in Subsection~\ref{subsec:equal-block-lifting}, and $|U|=q^{2ms}$. Hence $U$ has odd order and therefore admits a complete mapping. Since $Q=U\rtimes X$, the product $Q=UX$ is an exact factorization. Thus Lemma~\ref{lem:exact-factorization} extends the complete mapping $\theta_X$ of $X$ to a complete mapping $\theta_Q$ of $Q$. Moreover, $[Q:X]=|U|$ is odd. Since $P\in\mathrm{Syl}_2(X)$, it follows that $P\in\mathrm{Syl}_2(Q)$. Also, since $\theta_Q|_X=\theta_X$, $\theta_Q(P)=\theta_X(P)=P$.

Finally, since $G$ acts transitively on the $d_1$-dimensional subspaces of $V$ and $Q=G_W$, the orbit--stabilizer theorem gives $[G:Q]=\binom{d_1+d_2}{d_1}_q$. Since $d_1$ and $d_2$ are distinct powers of $2$, Lemma \ref{lem:qbinom-parity}-(2) shows that $[G:Q]$ is odd. Applying Proposition~\ref{prop:one-step-flag-lifting} to the refinement $0<V$ to $0<W<V$, the complete mapping $\theta_Q$ extends to a complete mapping $\theta$ of $G$. Since $P\in\mathrm{Syl}_2(Q)$ and $[G:Q]$ is odd, we have $P\in\mathrm{Syl}_2(G)$. Moreover, $\theta(P)=\theta_Q(P)=P$. Hence $\mathsf{CM}(n)$ holds.
\end{proof}

Combining Proposition~\ref{pro:power-two}, Proposition \ref{prop:general-r-ge-3}, and Proposition \ref{prop:general-r-2}, we conclude that $\mathsf{CM}(n)$ holds for every $n\geq2$.

\section{Over finite fields of even characteristic}\label{sec:even}

In this section, assume that $n\geq2$ and that $q$ is even. We retain the notation $G=\SL_n(q)/Z$, $Z\leq Z(\SL_n(q))$, and let $\pi:\SL_n(q)\longrightarrow G$ be the natural quotient map.

\begin{pro}\label{prop:even-characteristic}
Assume that $n\geq2$, $q$ is even, and $(n,q)\neq(2,2)$. Then $G$ admits a complete mapping stabilizing a Sylow $2$-subgroup.
\end{pro}

\begin{proof}
Let $\mathcal F:0<V_1<\cdots<V_{n-1}<V$ be the complete flag whose stabilizer in $\SL(V)$ is $\widehat B=\widehat U\rtimes\widehat T^{+}$, where $\widehat B$, $\widehat U$, and $\widehat T^{+}$ are as in Definition~\ref{def:Borel}. Let $B:=G_{\mathcal F}$. Then $B=\pi(\widehat B)$. Set $U:=\pi(\widehat U)$ and $T:=\pi(\widehat T^{+})$. Since $Z\leq\widehat T^{+}$, the decomposition $\widehat B=\widehat U\rtimes\widehat T^{+}$ induces $B=U\rtimes T$. Moreover, $|U|=q^{\frac{n(n-1)}{2}}$, so $U$ is a $2$-group since $q$ is a power of $2$. On the other hand, $|\widehat T^{+}|=(q-1)^{n-1}$ and $q-1$ is odd. Hence $T=\pi(\widehat T^{+})$ has odd order. Therefore $U\in\mathrm{Syl}_2(B)$.

We claim that $U$ is noncyclic. If $q>2$, then $\{\pi(I_n+aE_{12}):a\in\F_q\}\cong(\F_q,+)$ is readily seen to be a subgroup of $U$. Since $q>2$ is a power of $2$, this is a noncyclic elementary abelian $2$-group. Hence $U$ is noncyclic. If $q=2$, then the assumption $(n,q)\neq(2,2)$ implies that $n\geq3$, and $\langle\pi(I_n+E_{12}),\pi(I_n+E_{13})\rangle\cong C_2^2$ is readily seen to be a subgroup of $U$. Hence $U$ is again noncyclic. Thus $U$ is noncyclic in all cases under consideration. By Theorem \ref{thm:HP}, $U$ admits a complete mapping. Since $T$ has odd order, $T$ admits a complete mapping.

Since $B=TU$ and $T\cap U=1$, the product $B=TU$ is an exact factorization. Hence Lemma~\ref{lem:exact-factorization} extends a complete mapping of $U$ to a complete mapping $\theta_B$ of $B$ such that $\theta_B(U)=U$.

It remains to lift $\theta_B$ from $B$ to $G$. Since $\SL(V)$ acts transitively on the complete flags of $V$ and $Z$ fixes every subspace of $V$, the induced action of $G$ on the complete flags is transitive. Since $B=G_{\mathcal F}$, the orbit--stabilizer theorem, together with successive choices of the subspaces in a complete flag, gives
\[
[G:B]=\prod_{k=1}^{n-1}\binom{n-k+1}{1}_q.
\]
Since $q$ is even, Lemma \ref{lem:qbinom-parity}-(1) shows that each factor in this product is odd. Therefore $[G:B]$ is odd.

The complete flag $\mathcal F$ is obtained from the trivial flag $0<V$ by inserting finitely many subspaces. Therefore Corollary~\ref{cor:flag-refinement-lifting} extends $\theta_B$ to a complete mapping $\theta$ of $G$. Since $U\in\mathrm{Syl}_2(B)$, $[G:B]$ is odd, and $\theta_B(U)=U$, the same corollary gives $U\in\mathrm{Syl}_2(G)$ and $\theta(U)=U$.
\end{proof}

The remaining case is exceptional.

\begin{rem}\label{rem:SL2-2}
If $(n,q)=(2,2)$, then $Z=1$ and $G=\SL_2(2)\cong S_3$. Its Sylow $2$-subgroups are cyclic of order $2$, so by Theorem \ref{thm:HP}, $G$ admits no complete mapping.
\end{rem}

\section{Summary and concluding remarks}\label{sec:8}

We are now ready to prove Theorem \ref{thm:main}.

\begin{proof}
The odd-characteristic case follows from Propositions~\ref{pro:power-two}, \ref{prop:general-r-ge-3}, and \ref{prop:general-r-2}. The even-characteristic case with $(n,q)\neq(2,2)$ follows from Proposition~\ref{prop:even-characteristic}, while the exceptional case $(n,q)=(2,2)$ is excluded by Remark~\ref{rem:SL2-2}.
\end{proof}

We conclude by recalling the connection between complete mappings and difference matrices. A $(G,k,1)$-\emph{difference matrix} (DM) is a $k\times |G|$ matrix $D=(d_{ij})$ with entries in a group $G$ such that for any two distinct row indices $i_1$ and $i_2$, the multiset of differences $\{d_{i_1j} d_{i_2j}^{-1}\mid 1\leq j\leq |G|\}$ contains each element of $G$ exactly once. In particular, a $(G,3,1)$-difference matrix exists if and only if $G$ admits a complete mapping. Recursive constructions for difference matrices were developed systematically by Buratti \cite{Buratti98}. In particular, \cite[Corollary~2.6]{Buratti98} shows that if $N\unlhd G$ and both $N$ and $G/N$ admit difference matrices with $k$ rows, then so does $G$. Thus, Lemma~\ref{lem:normal-extension} can be viewed as
a refinement of the $k=3$ case of Buratti's construction, with
the additional requirement that the resulting complete mapping
stabilize a Sylow $2$-subgroup.

A natural direction for further research is to extend our results to the remaining finite simple groups of Lie type and to the sporadic simple groups appearing in Theorem~\ref{cor:main}.
For the remaining classical groups, the flag-stabilizer and decomposition methods developed here suggest analogous constructions involving totally isotropic or totally singular flags and decompositions compatible with the underlying forms. The main challenge is to identify suitable subgroups of odd
index while retaining control of a Sylow $2$-subgroup.

\section*{Conflict of interest}

The authors declare no competing interests.

\section*{Data availability}

No datasets were generated or analysed during the current study.

\end{document}